\documentclass[a4paper, 11 pt , english]{article}

\usepackage[english]{babel}
\usepackage{ulem}
\usepackage[utf8]{inputenc}
\usepackage[T1]{fontenc} 
\usepackage{fancybox} 
\usepackage{alltt}
\usepackage{graphicx}
\usepackage{float}
\usepackage{bm}
\usepackage{lmodern} 
\usepackage{babel}
\usepackage[top = 3cm, bottom = 3 cm, right = 3cm, left = 3 cm]{geometry}
\usepackage{appendix}
\usepackage{bigints}
\usepackage{pgf,tikz,pgfplots}
\pgfplotsset{compat=1.15}
\usepackage{mathrsfs}
\usepackage{dsfont}
\usetikzlibrary{arrows}
\usepackage{graphicx}
\usepackage{amsmath}
\usepackage{amssymb}
\usepackage{mathtools} 
\usepackage[T1]{fontenc} 
\usepackage{lmodern} 
\usepackage[utf8]{inputenc} 
\usepackage{babel}
\usepackage[linesnumbered, french]{algorithm2e}
\usepackage{physics}
\usepackage{amsthm}
\usepackage{mathtools}
\usepackage{empheq}
\usepackage[most]{tcolorbox}
\usepackage{stmaryrd}
\usepackage{enumitem}
\newtheorem{Theorem}{Theorem}[section]
\newtheorem{Definition}[Theorem]{Definition}
\newtheorem{Proposition}[Theorem]{Proposition}

\newtheorem{Lemma}[Theorem]{Lemma}
\newtheorem{Corollary}[Theorem]{Corollary}
\newtheorem{Remark}[Theorem]{Remark}

\newtheorem*{AssumptionA1}{Assumption $\textbf{\rm A}_{\mathbb{W}_p}$}
\newtheorem*{AssumptionA2}{Assumption $\textbf{\rm A}_{\mathbb{W}^{\rm ad}_p}$}
\newtheorem*{AssumptionB1}{Assumption $\textbf{\rm B}_{\mathbb{W}_p}$}

\newtheorem*{AssumptionC1}{Assumption ${\rm C}_{ \mathbb{W}_p} $}
\newtheorem*{AssumptionC2}{Assumption ${\rm C}_{ \mathbb{W}^{\rm ad}_p}$}
\newtheorem*{AssumptionD1}{Assumption ${\rm D}_{ {\rm M} }$}

\def \Sum{\displaystyle\sum}

\def \be{\begin{eqnarray}}
\def \ee{\end{eqnarray}}
\def \b*{\begin{align*}}
\def \e*{\end{align*}}

\def \E{\mathbb{E}}

\def \L{\mathbb{L}}
\def \N{\mathbb{N}}
\def \P{\mathbb{P}}

\def \R{\mathbb{R}}
\def \TT{{\bf T}}

\def \W{\mathbb{W}}
\def \X{\mathbb{X}}

\def \[{[\,\!\![}
\def \]{]\,\!\!]}

\def \1{{\bf 1}}

\def \ep{\hbox{ }\hfill$\Box$}

\def\Ec{{\cal E}}

\def\Lc{{\cal L}}

\def\Kc{{\cal K}}

\def\Pc{{\cal P}}

\def\namedlabel#1#2{\begingroup
    #2%
    \def\@currentlabel{#2}%
    \phantomsection\label{#1}\endgroup
}
\newcommand{\vertiii}[1]{{\left\vert\kern-0.25ex\left\vert\kern-0.25ex\left\vert #1 
  \right\vert\kern-0.25ex\right\vert\kern-0.25ex\right\vert}}

\let\L\undefined
\newcommand{\L}{\mathbb{L}}

\newlist{romanenumerate}{enumerate}{1}
\setlist[romanenumerate,1]{label=\textup{(\roman*)}, leftmargin=2em}

\usepackage[colorlinks,hyperindex,bookmarks,linkcolor=blue,citecolor=blue,urlcolor=blue]{hyperref}

\usepackage{wrapfig}
\usepackage{epsf}
\usepackage{biblatex}
\title{\bf Higher-Order Sensitivity Analysis of Wasserstein and Adapted Wasserstein Distributionally Robust Optimization Problems}
\author{            Nathan Sauldubois\thanks{
            New York University, Tandon School of Engineering
                                ns6982@nyu.edu
                                } 
}
\date{\today} 

\DefineBibliographyStrings{english}{
 and = {\&}}

\begin{document}

\maketitle

\begin{abstract}
We investigate the higher-order expansions of the sensitivities of functionals of probability measures. Building on the work of \citeauthor{bartl2021sensitivity} \cite{bartl2021sensitivity} and \citeauthor{bartlsensitivityadapted} \cite{bartlsensitivityadapted}, who established the first-order expansion when the functional arises from a stochastic optimisation problem, we extend these results using recent developments in differential calculus on Wasserstein spaces. Specifically, we generalise their approach to broader classes of functionals defined on the Wasserstein space. We provide an explicit formula for the second-order expansion of the sensitivity of a functional with respect to both the Wasserstein and the adapted Wasserstein metrics. Furthermore, we propose a construction method for the higher-order expansion of these sensitivities. Finally, following recent advances by \citeauthor{Touzisauldubois2024ordermartingalemodelrisk} \cite{Touzisauldubois2024ordermartingalemodelrisk} and \citeauthor{JiangObloj} \cite{JiangObloj} concerning sensitivity of Distributionally Robust Optimisation under martingale constraints, we also derive higher-order expansions for sensitivities in the presence of such constraints.
\end{abstract}

\noindent \textbf{2020 Mathematics Subject Classification:} 49K45, 49Q22, 90C31, 91G70.

\section{Introduction}

Consider a functional $g : \mathcal{P}_p(E) \to \mathbb{R}$, where $\mathcal{P}_p(E)$ denotes the set of probability measures on a Polish space $E$ with finite $p$-th moments. We are interested in the sensitivity of $g(\mu)$ with respect to variations of the underlying measure $\mu$. Computing the value of $g(\mu)$ encompasses a wide range of problems, including stochastic optimisation, optimal transport, and pricing financial contracts. Understanding how $g(\mu)$ reacts to perturbations of $\mu$ is crucial for addressing model uncertainty, often referred to as Knightian uncertainty, introduced by \citeauthor{knight1921risk} \cite{knight1921risk}. Knightian uncertainty accounts for the fact that an agent cannot perfectly know or predict their environment and is therefore exposed to potential model misspecifications. 
Distributionally Robust Optimisation (DRO) addresses the model misspecification problem by optimising against worst-case scenarios over a suitably chosen family of models. As emphasised by \citeauthor{gao2023distributionally} \cite{gao2023distributionally}, the ambiguity set must be selected carefully, balancing between the need to reflect the intrinsic features of the problem and the requirement of mathematical tractability. Consequently, the choice of the ambiguity set heavily depends on the nature of the functional $g$ and the underlying space $E$.
For instance, when $E$ is the space of continuous functions, a natural choice for the family of models is the set of semimartingale laws with uncertain drift and volatility characteristics, as considered in \citeauthor{bartl2024numerical} \cite{bartl2024numerical}, \citeauthor{bartl2023sensitivity} \cite{bartl2023sensitivity}, \citeauthor{bartl2021duality} \cite{bartl2021duality}, 
\citeauthor{compoint2025sensitivity} \cite{compoint2025sensitivity} and 
\citeauthor{geuchen2022affine} \cite{geuchen2022affine}.

A natural choice for the ambiguity set is to consider measures that are close to a reference measure according to a suitable distance or divergence. For instance, when using the Kullback-Leibler divergence, this approach has been extensively studied in \citeauthor{lam_robust_2016} \cite{lam_robust_2016}, \citeauthor{lam2018sensitivity} \cite{lam2018sensitivity}, and \citeauthor{petersen2000minimax} \cite{petersen2000minimax}.
The choice of the Kullback-Leibler divergence offers tractability, and as demonstrated in \citeauthor{lam_robust_2016} \cite{lam_robust_2016}, it leads to an explicit second-order expansion of the robust objective in the unconstrained setting. This has recently been pushed to higher orders and connected to Bartlett-type coverage correction for general von Mises differentiable functionals by \citeauthor{he_higher-order_2025} \cite{he_higher-order_2025}.
Alternatively, one may define the ambiguity set as a ball of radius $r$ with respect to the Wasserstein distance, or, in dynamic settings, with respect to the adapted Wasserstein distance. The adapted Wasserstein distance was introduced to restore stability of stochastic optimisation problems under adapted perturbations by \citeauthor{backhoff_adapted_2020} \cite{backhoff_adapted_2020}, and refines the classical Wasserstein distance by additionally requiring causality of the optimal coupling \cite{backhoff_all_2020, eder2019compactness}; see also \citeauthor{backhoff_estimating_2022} \cite{backhoff_estimating_2022} and \citeauthor{bartl_wasserstein_2021} \cite{bartl_wasserstein_2021} for further developments, and \citeauthor{pflug_multistage_2014} \cite{pflug_multistage_2014} for the related nested distance in the stochastic programming literature. \citeauthor{blanchet_quantifying_2016} \cite{blanchet_quantifying_2016} introduced an ambiguity set based on an optimal transport criterion, and studied strong duality of the DRO problem. The study of Wasserstein DRO has been done under two main perspectives: the first is duality, which is the point of view taken by \citeauthor{blanchet_quantifying_2016} in \cite{blanchet_quantifying_2016}, \citeauthor{zhang2025short} \cite{zhang2025short}, and \citeauthor{gao2023distributionally} \cite{gao2023distributionally}. The other paradigm to study this problem is the sensitivity analysis, that is to say, finding the expansion of the error with respect to the radius $r$ of the ball, or with respect to a rescaling parameter (in the case of penalisation formulation, see \citeauthor{JiangObloj} \cite{JiangObloj}, or \citeauthor{nendel_parametric_2022} \cite{nendel_parametric_2022}). This perspective is what motivated recent works such as \citeauthor{bartl2021sensitivity} \cite{bartl2021sensitivity} for the sensitivity with respect to the Wasserstein distance, \citeauthor{bartlsensitivityadapted} \cite{bartlsensitivityadapted} for the adapted Wasserstein distance in discrete time and \citeauthor{JiangObloj} \cite{JiangObloj} for the adapted Wasserstein sensitivity with martingale constraints in discrete and continuous time and \citeauthor{Touzisauldubois2024ordermartingalemodelrisk} \cite{Touzisauldubois2024ordermartingalemodelrisk} for the Wasserstein distance and the adapted Wasserstein distance under martingale constraints. In a regret-minimisation variant of Wasserstein DRO, \citeauthor{fiechtner_wasserstein_2025} \cite{fiechtner_wasserstein_2025} show that the first-order sensitivity vanishes at the empirical risk minimiser, so that a second-order expansion governs the local behaviour of the regret. A related but distinct line of work, \citeauthor{lanzetti_first-order_2022} \cite{lanzetti_first-order_2022}, derives first-order optimality conditions for optimisation problems posed directly on the Wasserstein space, rather than expansions of the value of a fixed functional with respect to the radius of an ambiguity ball, which is the perspective adopted here.

The focus of the present paper is the sensitivity of DRO with respect to the adapted Wasserstein and the classical Wasserstein distance. Our goal is to refine and extend the asymptotic expansions obtained in \citeauthor{bartl2021sensitivity} \cite{bartl2021sensitivity} and \citeauthor{bartlsensitivityadapted} \cite{bartlsensitivityadapted}, by developing explicit second-order formulas and providing a method for constructing higher-order expansions in Proposition \ref{prop:approx 1} and \ref{prop:expansion order n}. To this end, we rely on the differential calculus on the space of probability measures, through the notion of linear functional derivative, as developed in the mean-field game literature by \citeauthor{CarmonaDelarue} \cite{CarmonaDelarue}, and further used by \citeauthor{chassagneux2022weak} \cite{chassagneux2022weak} and \citeauthor{jourdain2021central} \cite{jourdain2021central}. Unlike \citeauthor{Touzisauldubois2024ordermartingalemodelrisk} \cite{Touzisauldubois2024ordermartingalemodelrisk}, we do not restrict ourselves to the martingale case, and we provide a sufficient condition under which the martingale sensitivity with respect to the adapted Wasserstein distance admits an order$-2$ expansion in Proposition \ref{prop:mart expansion order 2}.
All these approximation results are employed in Section \ref{sec:Numerical illustration} to assess the accuracy of the expansions. Furthermore, building on \citeauthor{nendel_parametric_2022} \cite{nendel_parametric_2022}, we show that incorporating order-2 insights into the initialisation of the neural network accelerates the solution of DRO problems under both Wasserstein and adapted Wasserstein criteria.

The remainder of the paper is organised as follows. Section \ref{sec:Notations and definitions} introduces the notations and definitions used throughout the paper. Section \ref{sec:Main results} presents our main results, namely the second-order expansions of the Wasserstein and adapted Wasserstein sensitivity, the construction of the order-$n$ expansion, and the sufficient condition for the martingale case. Section \ref{sec:Numerical illustration} illustrates these results numerically and studies the resulting acceleration of the neural network initialisation. Section \ref{sec:Proofs} collects the proofs of our main results.

\section{Notations and definitions}\label{sec:Notations and definitions}

Throughout this paper, let $S$ be a Euclidean space endowed with the norm $ \vert \cdot \vert $. As usual, for $p >1$ we denote by $p'$ the conjugate exponent, \textit{i.e.}, $ \frac{1}{p} + \frac{1}{p'} = 1$. 
Let $\Pc (S) $ be the collection of all probability measures $\mu$ on $E$, with the corresponding subset of finite $p-$th moment measures:
$$
\Pc_p ( S ):=\Big\{\mu\in\Pc ( S ) :~\E^\mu[|X|^p]<\infty\Big\},
~~\mbox{for all}~p\ge 1.
$$
We introduce the projection coordinates $(X,X')$ on $\X := S \times S$ defined by $X(x,x')=x$ and $X'(x,x')=x'$ for all $ (x,x' )\in \X$. For $\mu,\mu'\in\Pc ( S ) \times \Pc ( S )$, we define the set of all couplings
$$
\Pi(\mu,\mu')
:=
\Big\{ \pi\in\Pc ( \X ) :\pi\circ X^{-1}=\mu~\mbox{and}~\pi\circ {X'}^{-1}=\mu'\Big\},
$$ 
The $p-$Wasserstein distance between $\mu$ and $\mu'$ is defined as:
\begin{eqnarray*}
\W_p(\mu,\mu')
:=
\inf_{\pi\in\Pi(\mu,\mu')}\E^\pi\Big[\big|X-X'\big|^p\Big]^{\frac1p},
&\mbox{for all}&
\mu,\mu'\in\Pc_p(S)
\end{eqnarray*}
We denote the corresponding ball of radius $r$ by
\begin{equation*} 
B_{\W_p} (\mu, r)
:=
\big\{\mu'\in\Pc(\X):\W_p(\mu,\mu')\le r \big\}
\end{equation*}
In the dynamic setting that we are addressing in this paper, we define the time steps on the state space $ U = S \times S$ for some Euclidean space $S$, through the projection coordinate maps on $\X \times \X$: 
$$ 
X_i (x, x') = x_i
~\mbox{and}~ 
X'_i (x,x') = x'_i,
~i=1, 2,
~\mbox{for all}~
(x, x')\in \X
.$$ 
Define ${\rm M} $ as the set of all probability measures on $S$ with finite $p$-th moment that are martingales, 
\begin{eqnarray*}
\rm M
&:=&
\Big\{\mu\in\Pc (S) :~\E^\mu[X_2|X_1]=X_1,~\mu-\mbox{a.s.}\Big\},
\end{eqnarray*}
and we recall that for $\mu\in \Pc_p (S)$, we have
\begin{eqnarray*}
\mu\in \rm M
&\mbox{if and only if}&
\E^\mu\big[h^\otimes \big]=0,
~\mbox{for all}~
h\in \mathbb{L}^{p'}(\mu_1),
\end{eqnarray*}
where we used the notation $
h^\otimes(x)
:=h(x_1)\cdot(x_2-x_1),
$ for all $
x=(x_1,x_2)\in \X.
$ 
\begin{Definition}
A probability measure $ \mathbb{P} \in \mathcal{P}( \X ) $ is causal if $\mathbb{F}^{X} := \sigma( X ) $ is compatible with $\mathbb{F}^{X'} := \sigma( X' ) $, in the sense that for all bounded Borel-measurable $ f : S \times S \rightarrow \R $ and $ g : S \rightarrow \R $,
$$
\E^{\P} \big[ 
f ( X_1, X_2 ) g ( X'_1 ) \vert X_1 \big] 
=
\E^{\P} \big[ 
f ( X_1, X_2 ) \vert X_1 \big] 
\,
\E^{\P} \big[ 
g ( X'_1 ) \vert X_1 \big] 
.
$$
\end{Definition}
We introduce the set of {\it bi-causal} couplings 
$$
\Pi^{\text{bc}}(\mu,\mu')
:=
\{\pi\in \Pi ( \mu, \mu' ) \,\, \text{such that} \,\, \pi \,\, \text{and} \,\,\pi \circ ( X', X )^{-1} \,\, \text{are causal } \},
$$
together with the corresponding {\it adapted Wasserstein} distance 
\begin{eqnarray*}
\W^{{\rm ad}}_p(\mu,\mu')
:=
\inf_{\pi\in\Pi^{bc}(\mu,\mu')}\E^\pi\Big[\big|X-X'\big|^p\Big]^{\frac1p},
&\mbox{for all}&
\mu,\mu'\in\Pc_p(U).
\end{eqnarray*}
We also introduce the set of causal mappings $ \mathbb{L}^p_{\rm ad} ( \mu ) := \{ T := ( T_1 (X_1) , T_2 (X) ) , \,\, T_1 \in \mathbb{L}^p( \mu_1) \, , \, T_2 \in \mathbb{L}^p ( \mu )\}$, endowed with the norm $ \Vert T \Vert^p_{\mathbb{L}^p_{ \rm ad} ( \mu )} := \Vert T_1 \Vert_{\mathbb{L}^p ( \mu_1)}^p + \Vert T_2 \Vert_{\mathbb{L}^p ( \mu)}^p $. We denote the ball of radius $r$ by
$
B_{\W^{\rm ad}_p} (\mu, r)
:=
\big\{\mu'\in\Pc_p(S):\W^{{\rm ad}}_p(\mu,\mu')\le r\big\}
$ and the ball restricted to the martingale measure 
\begin{equation}\label{eqdef:Mart ad wass ball}
 B^{\rm M}_{\W^{\rm ad}_p} (\mu, r) := B_{\W^{\rm ad}_p} (\mu, r) \cap {\rm M}
.\end{equation}
Throughout this paper, we consider a function $g:\Pc_p(S)\longrightarrow\R$ with appropriate smoothness in the following sense. 
\begin{Definition}
We say that $g$ has a linear functional derivative if there exists a continuous function $\delta_mg:\Pc_p(S) \times S \longrightarrow\R$, with $p-$polynomial growth in $x$, locally uniformly in $m$, such that for all $\mu,\mu'\in\Pc_p(S)$, and denoting $\bar\mu_\lambda:=\mu+\lambda(\mu'-\mu)$, we have
$$
\frac{g\big(\bar\mu_\lambda\big)-g(\mu)}{\lambda}
\longrightarrow
\langle\delta_mg(\mu,x),\mu'-\mu\rangle
:=
\int_S \delta_mg(\mu,x)(\mu'-\mu)(\mathrm{d}x),
~\mbox{as}~\lambda\searrow 0.
$$
\end{Definition}
\noindent Recursively, we can define the $k$-th order linear functional, satisfying
\be
\delta^{k-1}_m g(\mu', y)- \delta^{k-1}_m g(\mu, y)
=
\int_0^1 \!\!\!\int_S \delta^{k}_mg(\bar\mu_\lambda,x,y )(\mu'-\mu)(\mathrm{d}x)d\lambda,
&
\mu,\mu'\in\Pc_p \big( S \big)
\ee
\noindent for all $y, y' \in S$. This definition of the linear functional derivative induces a Taylor's formula (Lemma $2.1$ of \citeauthor{chassagneux2022weak} \cite{chassagneux2022weak}). Given two normed vector spaces $(E, \Vert \cdot \Vert_E)$ and $(F, \Vert \cdot \Vert_F)$, define 
$$ 
{\rm Pol}_p (E, F)
:=
\Big\{ g : E \rightarrow F \,\, : \,\, \exists C > 0, \,\,
\Vert g(x) \Vert_F \leq C ( 1 + \Vert x \Vert_E^p ) \,\, \text{for all } x \in E
\Big\}
.$$
Whenever the context is clear, we will omit $E$ and $F$.

\begin{Proposition}\label{prop:Taylor expansion measure order n}
Let $ g : \Pc_p ( S ) \rightarrow \R $ be an $n-$times linearly differentiable map, such that $\delta_m^j g (\mu, \cdot) \in {\rm Pol}_p $, locally uniformly in $\mu$, and is measurable. Then, the following identity holds
\begin{equation*}
\begin{split}
 g(\mu')
=
 g(\mu)
&+
\sum_{k = 1}^{n-1} 
\frac{1}{k!}
\int_{ S^k} 
\delta^{k}_m g(\mu, x )(\mu'-\mu)^{\otimes k} (dx)
\\
&+
\int_0^1 \frac{ ( 1 - \lambda )^{n-1 }}{(n-2)!} 
\int_{S^{n}}
\delta^{n}_m g( \bar{\mu}_\lambda , x )(\mu'-\mu)^{ \otimes n} (dx)
\mathrm{d}\lambda
\end{split}
\end{equation*}
for all $
\mu,\mu'\in\Pc_p ( S )$, where $ \bar{\mu}_\lambda = ( 1 - \lambda ) \mu + \lambda \mu '$.
\end{Proposition}

\noindent Another useful result from Example $2.7$ \citeauthor{jourdain2021central} \cite{jourdain2021central} is the symmetry property of the $k-$th linear functional derivative.

\begin{Proposition}\label{prop:symm k-th lin deriv}
Let $ g : \Pc_p \big( S \big) \rightarrow \R $ be an $n-$times linearly differentiable map, such that $\delta_m^j g (\mu, \cdot) \in {\rm Pol}_p $, locally uniformly in $\mu$, and is measurable. Then, $\delta_m^j g (\mu, \cdot)$ is symmetric, that is to say, the following equality holds
\begin{equation*}
\begin{split}
\delta_m^j g (\mu, x)
=
\delta_m^j g (\mu, x_{\sigma})
\,\, \text{for all  $x \in S^j$ and $\sigma \in \mathfrak{S}_j$.}
\end{split}
\end{equation*}
where $ x_\sigma = ( x_{\sigma(i)} )_{ 1 \leq i \leq  j } $.
\end{Proposition}

We now introduce the gradient of a Fréchet differentiable function $F : \L^p(\mu) \rightarrow \R $.

\begin{Definition}\label{def:gradient_frechet}
Given a map $F : \mathbb{L}^p (\mu) \rightarrow \R$ for $p > 1$, which is Fréchet differentiable, we denote by $ \nabla F $ the unique element in $\mathbb{L}^{p'} (\mu)$ such that for $T, h \in \mathbb{L}^p (\mu)$, $F(T + h ) = F(T) + \E^{\mu} [ \nabla F (T) h ]+ \circ (\Vert h \Vert_{\mathbb{L}^p (\mu)}) $. 
\end{Definition}

\noindent 
In the following, for a function $ f : \X^k \rightarrow \R$, we define 
\begin{equation}\label{eqdef:copies}
\E^{ \mu^{\otimes k}} [ f(X)] := \int_{\X^k} f(x_1, \cdots, x_k) \mu(\mathrm{d}x_1) \otimes \cdot \otimes \mu(\mathrm{d}x_k) 
\end{equation}
which is the expectation of $k$ independent copies. 

\noindent For clarity, we will only deal with the case where $ E_{\W_p} = \R$ for the classical Wasserstein case and $ E_{\W^{\rm ad}_p} = \R \times \R $ for the adapted Wasserstein case.
All results extend to the multi-dimensional case using the corresponding tools of differential calculus.
We will also use the notation 
\begin{equation}\label{eqdef:causal gradient}
\partial_x^{\rm ad} \delta_m g 
:= 
\begin{bmatrix} 
\E ^\mu_1 [ \partial_{x_1} \delta_m g ]
\\
\partial_{x_2} \delta_m g 
\end{bmatrix}
.\end{equation}
 We set the two differential operators $\partial_x^{\W_p} = \partial_x$ and $ \partial_x^{\W_p^{\rm ad}} := \partial_x^{\rm ad} $.

 \noindent Now, set 
\begin{equation}\label{eqdef:w_p}
 w_p(x) := \text{sgn} (x) x^{p - 1}
\end{equation}
For $x \in \R$ and $y \in \R$, we have 
\begin{equation}\label{eqdef:definition w_p and equivalence}
w_p(x) = y \text{  if and only if  } w_{p'} (y) = x. 
\end{equation} 
Also, let $ k \in \llbracket 0 , n \rrbracket $, $ i = ( i_1, \cdots, i_k ) \in \mathbb{N}^k$ and $ F$ be a function in $C^k ( \R^\ell, \R )$.
For $ x \in \R^\ell $, define:
\begin{equation}\label{eqdef:diff_operator}
\partial^{\vert i \vert_1}_{ \mathbf{x}^i} 
:=
\partial^{\vert i \vert_1}_{ \Pi_{j =1}^k ( x^j )^{i_j}} 
.
\end{equation}
Let $X$ be a vector space, $r \in \R$ and $ n \in \N$. 
For $\mathbf{x} = (x_1, \cdots, x_n) \in X^n $, and $k \leq n$, we set 
\begin{equation}\label{eqdef:sommepondere}
S_k(r, \mathbf{x} ) = \sum_{i = 1}^k r^{i-1} x_i.
\end{equation}
Now, let $ F : X \rightarrow Y $ where $ X $ and $Y$ are vector spaces, define for $x \in X $ and $ h \in X $ 
\begin{equation}\label{eqdef:bracket}
\Delta_h [ F ] ( x ) 
:=
F ( x + h ) - F ( x ) 
.
\end{equation}
For a real number $p$ and $ i \in \mathbb{N}^k$, set 
\begin{equation}\label{eqdef:conjugate_exponent}
p_i := \frac{p}{ \vert i \vert_{\infty} }\,\,\text{if} \,\, i \neq 0 \,\, \text{and} \,\, p_i = \infty \,\, \text{otherwise}. 
\end{equation} 
Note that for $ \vert i \vert_1 < p $, we have $1 < p_i \leq \infty $. 
Finally, define
\begin{equation}\label{eqdef:def_Eell}
\mathcal{E}_\ell
:=
\lbrace 
h \in C^{\ell} ( \R, \R )
\,\,
\text{such that}
\,\,
h^{ ( k ) } \in {\rm Pol}_{p - k - 1}
\,\,
\text{for all} 
\,\,
0 \leq k \leq \ell
\rbrace
.
\end{equation}

\section{ Main results }\label{sec:Main results}

Let $1 <p$, and let $ 1 < p' < \infty $ be its conjugate exponent. Fix $\mu \in \Pc_p \big( E_{\mathbf{d}} \big)$, where $ \mathbf{d} $ is either $ \W_p $ or $ \W^{\rm ad}_p$. As mentioned before, we want to study the order $n$ expansion of 
$$
G_{ \mathbf{d} } (r) := \sup_{ \mu' \in B_\mathbf{d} ( \mu, r) } g( \mu' )
.$$ 
This section is divided into two parts. First, we derive the order-$n$ expansion for the classical and the adapted Wasserstein DRO. Second, we use those results to obtain explicit second-order expansions. We then provide a sufficient condition under which the martingale distributionally robust optimisation admits a second-order expansion.

\subsection{ Order $n$ expansion of the DRO under Wasserstein and adapted Wasserstein constraints.}

\begin{AssumptionA1} \label{ass:reformulation 1 wasserstein}
The measure $\mu$ is atomless, and the function $ g : \Pc( \R ) \rightarrow \R $ is continuous with respect to $\W_p$.
\end{AssumptionA1}

\begin{AssumptionA2} \label{ass:reformulation 1 adapted wasserstein}
The marginal $\mu \circ X_1^{-1} $ is atomless and the function $ g : \Pc( \R \times \R ) \rightarrow \R $ is continuous with respect to $ \W_p^{\rm ad}$.
\end{AssumptionA2}

\begin{AssumptionB1} \label{ass:reformulation 2 gen dist}
$\mu$, $g$, $n$, and $p$ satisfy
\begin{enumerate}[label=\textnormal{(\roman*)}]
\item \label{cond:regularity on g dist} For all $ 1 \leq \ell \leq n $, $ \delta^\ell_m g (\mu, \cdot) $ is $C^n$, with $ \partial^k_y \delta^\ell_m g$ $\mathbf{d}-$continuous.
\item \label{cond:estimate diff gen dist} For $ i \in \mathbb{N}^k$, 
$ \partial^{\vert i \vert_1}_{\mathbf{x}^i} \delta^{\ell}_m g \in {\rm Pol}_{ \vert i \vert_\infty \frac{p}{p'}}$.
\end{enumerate}
\end{AssumptionB1}

\begin{AssumptionC1} \label{ass:exist approx saddle point wasserstein}
There exists $ \kappa > 0 $ such that the mapping $g$ satisfies
\begin{enumerate}[label=\textnormal{(\roman*)}]
\item \label{cond:convexity wasserstein} 
$
\sup_{ \Vert T \Vert_{\mathbb{L}^p (\mu) } \leq 1 } \Big\{
\mathbb{E}^{\mu} \Big[ \big(D^2 \delta_m g \big) T^2 \Big]
+
\mathbb{E}^{\mu^{\otimes 2}} \Big[ \big( \partial^2_{x^1 x^2} \delta^2_m g \big) T ( X^1) T ( X^2 ) \Big] \Big\}
\leq 
-\kappa
.$
\item \label{cond:div by gradient wasserstein} $ \vert \partial_x \delta_m g (\mu, \cdot ) \vert \geq \kappa $ for some $\kappa > 0$, $\mu-$almost surely. 
\end{enumerate}
\end{AssumptionC1}

\begin{AssumptionC2} \label{ass:exist approx saddle point adapted wasserstein}
There exists $ \kappa > 0 $ such that the mapping $g$ satisfies
\begin{enumerate}[label=\textnormal{(\roman*)}]
\item \label{cond:convexity adapted wasserstein} $
\sup_{ \Vert T \Vert_{\mathbb{L}^p_{\rm ad} (\mu) } \leq 1 } \Big\{
\mathbb{E}^{\mu} \Big[ T \cdot \big( D^2 \delta_m g \big) T \Big]
+
\mathbb{E}^{\mu^{\otimes 2}} \Big[ \big( \partial^2_{x^2 x^1} \delta^2_m g \big) T ( X^1) \cdot T ( X^2 ) \Big] \Big\}
$
.
\item \label{cond:div by gradient adapted wasserstein} $ \vert \partial^{\rm ad}_x \delta_m g (\mu, \cdot ) \vert \geq \kappa $ for some $\kappa > 0$, $\mu-$almost surely. 
\end{enumerate}
\end{AssumptionC2}

\begin{Remark}
{
\rm 
We now study the implications of these assumptions in the linear case $ g(\mu) := \int f(x) \mu(\mathrm{d}x)$.

\noindent $\bullet$ Assumptions \hyperref[ass:reformulation 1 wasserstein]{${\rm A}_{\W_p}$}, \hyperref[ass:reformulation 1 adapted wasserstein]{${\rm A}_{\W^{\rm ad}_p}$} are clearly satisfied whenever $f \in {\rm Pol}_p$ and is continuous.

\noindent $\bullet$ Assumptions \hyperref[ass:reformulation 2 gen dist]{${\rm B}_{\mathbf{d}}$} are satisfied whenever $f \in C^n$, with $ D^{i} f \in {\rm Pol}_{p-i}$, where $ D^{i}$ is the differential of order $i$.

\noindent $\bullet$ Assumptions \hyperref[ass:exist approx saddle point wasserstein]{${\rm C}_{\W_p}$} or \hyperref[ass:exist approx saddle point adapted wasserstein]{${\rm C}_{\W^{\rm ad}_p}$} are satisfied whenever $f$ is $\kappa$ strongly concave on the support of $\mu$.
}
\end{Remark}
\noindent Define the following spaces, $\mathbb{L}^p_{\W_p} ( \mu ):= \mathbb{L}^p ( \mu ) $ and $ \mathbb{L}^p_{\W^{\rm ad}_p} ( \mu ) := \mathbb{L}^p_{\rm ad} (\mu)$. Define also the following polynomial operators, for $ T \in \mathbb{L}^p_{\W_p} ( \mu) $ 
\begin{equation}\label{eqdef:operator polyn}
L_{\W_p}^{n, r} (T) := 
g(\mu) + 
\sum_{\ell = 1}^{n} \frac{1}{\ell !}
\hat{\Lc}^{r, \ell}_{\W_p} (T)
\end{equation}
where $\hat{\Lc}^{r, \ell}_{\W_p} (T) 
:=
\sum_{ k = \ell }^{n} 
\sum_{  \substack{i \in (\mathbb{N}^*)^\ell \\ \vert  i \vert_1 = k}}
\frac{r^k}
{
 i !
}
\E^{\mu^{\otimes \ell}} \big[
\partial^{ k }_{\mathbf{x}^i} \delta^{\ell}_m g ( \mu, X ) \prod_{j=1}^{\ell} T ( X^j )^{i_j} \big]$ and 
$i! = \Pi_{j = 1}^{\ell} i_j $. For $ T = ( T_1, T_2 ) \in \mathbb{L}_{ \W^{\rm ad}_p}^p ( \mu ) $, 
$$L_{\W_p^{\rm ad}}^{n, r} (T)
:=
g(\mu) + 
\sum_{\ell = 1}^{n} \frac{1}{\ell !}
\hat{\Lc}^{r, \ell}_{\W^{\rm ad}_p} (T)
$$
where $$\hat{\Lc}^{r, \ell}_{\W^{\rm ad}_p} (T) :=
\sum_{ k = \ell}^{n} 
\sum_{ \substack{i \in ( \mathbb{N}^*)^\ell 
\\ 
\vert i \vert_1 = k}}
\sum_{ \substack{ 
j \in \N^\ell \\ j \leq i }} 
\frac{ r^k }{\ell! i!}
\E^{\mu^{\otimes\ell}} \Big[ \big( 
\partial^{k}_{ \mathbf{x}_1^j \mathbf{x}_2^{i-j} }
 \delta_m^\ell g 
 \big)
 \prod_{ m = 1 }^{\ell} 
 T_1 ( X^m _1 )^{j_m} 
 T_2 ( X^m )^{i_m - j_m}
\Big]
.$$

\begin{Proposition}\label{prop:approx 1}
Let $ \mathbf{d} = \W_p$ or $\W_p^{\rm ad}$. Assume that $g$ and $\mu$ satisfy Assumption \hyperref[ass:reformulation 1 wasserstein]{${\rm A}_\mathbf{d}$}, then we have
\begin{equation}\label{eq:first reform gen dist}
G_{ \mathbf{d} }(r) 
=
\sup_{ \substack{ T \in \mathbb{L}^p_{\mathbf{d}} ( \mu ) \\ \Vert T \Vert_{\mathbb{L}^p_{\mathbf{d}} ( \mu )} \leq 1 } } g( \mu^{ r, T} ) \,\, \text{where} \,\, \mu^{ r, T} := \mu \circ ( X + r T )^{-1}
.\end{equation}
Furthermore, if $g$ and $\mu$ also satisfy Assumption \hyperref[ass:reformulation 2 gen dist]{${\rm B}_\mathbf{d}$}, we have
\begin{equation}\label{eq:formulation 2 gen dist}
G_{ \mathbf{d} } (r) 
 =
 \sup_{ \substack{T \in \mathbb{L}^p_{\mathbf{d}} ( \mu ) \\ \Vert T \Vert_{\mathbb{L}^p_{\mathbf{d}} ( \mu ) } \leq 1} } L^{n, r}_{\mathbf{d}} \big( T \big)  
+
\circ ( r^n ) 
\end{equation}
where $\circ ( r^n ) $ does not depend on $T$.
\end{Proposition}

\begin{Proposition}\label{prop:expansion order n}
Assume that $g$, $ \mu $, $n$ and $p$ satisfy Assumptions \hyperref[ass:reformulation 1 wasserstein]{${\rm A}_\mathbf{d}$}, \hyperref[ass:reformulation 2 gen dist]{${\rm B}_\mathbf{d}$} and \hyperref[ass:exist approx saddle point wasserstein]{${\rm C}_\mathbf{d}$}. There exists $ \TT \in \big( \mathbb{L}^p_{\mathbf{d}} ( \mu ) \big)^n$ and $ \mathbf{ \lambda} := (\lambda_1, \cdots, \lambda_n) \in \R^n$ such that $ \lambda_1 > 0 $, the following expansion holds
\begin{equation*}
G_{ \mathbf{d} }( r) 
= 
L^{n,r}_{\mathbf{d}} \big( S_{n} (r, \TT ) \big)
+
\circ ( r^n )
,\end{equation*}
and 
\begin{equation}\label{eq:FOC gen dist T and lambd}
\text{for} \,\, 1 \leq k \leq n \,\,, \,\,
\left\{
 \begin{array}{ll}
 \nabla_T
L^{n, r}_{\mathbf{d}} \big( S_{k}(r, \TT ) , r S_k(r, \mathbf{\lambda} ) \big) - p w_p \big( S_{k}(r, \TT ) \big)
=
\circ_{{\mathbb{L}^{p'} ( \mu ) }} ( r^{k} ) \\
\big\Vert S_{k}(r, \TT ) \big\Vert^p_{\mathbb{L}^p_{\mathbf{d}} ( \mu )}
=
1 
+
\circ ( r^{k-1} ) 
 \end{array}
\right.
\end{equation}
where $w_p$ is defined by \eqref{eqdef:w_p} and $\nabla_T$ is the Fréchet gradient of the map $L$ with respect to $T$, which can be identified as an element of $ \mathbb{L}^{p'}_{\mathbf{d}} (\mu)$. 
\end{Proposition}

\subsection{Explicit computation of second-order expansions.}

All previous formulas leave little room for an explicit expression. Here, we present the second-order expansion of the DRO.

\begin{Proposition}\label{prop:order 2 expansion gen dist}
Let $ \mathbf{d} \in \{ \W_p, \W_p^{\rm ad} \} $ and let $g$, $\mu$, and $p$ satisfy Assumptions \hyperref[ass:reformulation 1 wasserstein]{${\rm A}_{ \W_p } $}, \hyperref[ass:reformulation 2 gen dist]{${\rm B}_{ \W_p }$} and \hyperref[ass:exist approx saddle point wasserstein]{${\rm C}_{ \W_p }$} for $n = 2$. Define 
\begin{equation}\label{eqdef:first order transport}
T^1_{\mathbf{d}}
:= \frac{ w_{p'} \big( \partial^{\mathbf{d}}_x \delta_m g \big) }{ \Vert \partial^{\mathbf{d}}_x \delta_m g \Vert_{\mathbb{L}^{p'}_{\mathbf{d}} ( \mu ) }^{p'/p} }
\end{equation}
where for $ z:= (z_1, z_2) \in \R^2$, $w_p(z) := \big( w_p(z_1), w_p(z_2) \big)$, $w_p$ is defined by \eqref{eqdef:w_p} and $\partial^{\mathbf{d}}$ is defined by \eqref{eqdef:causal gradient}. We have, 
\begin{align*}
G^{\mathbf{d}} ( r ) 
&=
L^{2, r}_{\mathbf{d}} \big( T^1_{\mathbf{d}} \big)
+
\circ ( r^2 ) 
.\end{align*}
\end{Proposition}

\begin{Remark}
{\rm
In the classical Wasserstein case $\mathbf{d} = \W_p$ and for a linear functional $g(\mu) = \int f(x, a) \mu(\mathrm{d}x)$, a second-order upper bound for $G^{\W_p}$ was already obtained in Remark $10$ of \citeauthor{bartl2021sensitivity} \cite{bartl2021sensitivity}, under a uniform growth bound on the Hessian of $f$. Proposition \ref{prop:order 2 expansion gen dist} refines this into an exact second-order asymptotic equality, valid for general functionals $g$ satisfying Assumptions \hyperref[ass:reformulation 1 wasserstein]{${\rm A}_\mathbf{d}$}, \hyperref[ass:reformulation 2 gen dist]{${\rm B}_\mathbf{d}$} and \hyperref[ass:exist approx saddle point wasserstein]{${\rm C}_\mathbf{d}$}, and extends it to the adapted Wasserstein distance.
}
\end{Remark}

\begin{Remark}
{\rm
The case $n=2$ is interesting because we have a simplification of the problem. Indeed, by studying the equation \eqref{eq:FOC gen dist T and lambd}, the norm condition implies that $ L^{1,r}(T^2) = 0$, simplifying the expression of the second-order development. This remark can be generalised to higher-order expansions, as the study of the equation \eqref{eq:FOC gen dist T and lambd} allows for some simplifications of the expansion expression given by Proposition \ref{prop:expansion order n}. However, even for $n = 3$, the expression becomes too intricate and hardly tractable. 
}
\end{Remark}

\subsection{Martingale case}

\noindent We now move on to the martingale case. Define 
$$
G^{\rm M}_{\rm ad}
:= 
\sup_{\mu' \in B^{\rm M}_{\W_p^{\rm ad}} (\mu, r) } g(\mu ')
$$
where $ B^{\rm M}_{\W_p^{\rm ad}} (\mu, r) $ is defined by \eqref{eqdef:Mart ad wass ball}. Now, for $ h_i \in \mathcal{E}_{n-i+1}$, $ 1 \leq i \leq n$, and $ T \in \mathbb{L}^p_{{\rm ad}} ( \mu )$, set $ \textbf{h}_n = ( h_1, \cdots, h_n ) $ and define 
\begin{equation}\label{eqdef:Gamm and L martingale}
\begin{split}
&\Gamma^{n, r} ( \textbf{h}_n , T ) 
:=
\sum_{k = 0}^{n-1} 
\sum_{j = 1}^{n-k}
\frac{r^{k + j}}{ k !} 
\E^{\mu} \Big[ 
h_j^{\left( k \right) } ( X_1 ) T_1( X_1 )^{k} \big( T_2 - T_1(X_1) \big) \Big]
\\
&\mathcal{L}^{n, r}_{\rm M } 
( T , \textbf{h}_n , \lambda ) 
:=
L^{n, r}_{\W_p^{\rm ad}} ( T ) 
+
\Gamma^{n, r} ( \textbf{h}_n , T) 
-
\lambda ( \Vert T \Vert^p_{\mathbb{L}^p_{{\rm ad} } (\mu)} - 1 ) 
.
\end{split}
\end{equation}

\begin{AssumptionD1} \label{ass: existence approx sad-point Mart adapted}
There exists $ \TT := ( \TT_1, \TT_2) \in \big( \mathbb{L}^p_{{\rm ad}} ( \mu ) \big)^n $, $ \mathbf{h}_n := ( h_1, \cdots, h_n ) $ such that $h_i \in \mathcal{E}_{n - i +1}$ and $ \lambda := (\lambda_1, \cdots, \lambda_n) \in \R^n$, with $ \lambda_1 > 0$, satisfying 
\begin{enumerate}[label=\textnormal{(\roman*)}]
\item \label{cond:saddle point mart adapt} 
\begin{equation*}
\begin{split}
&\nabla_{T} 
\mathcal{L}^{n, r}_{\rm M} 
\big( S_n(r, \TT) , \textbf{h}_n, r S_n(r, \mathbf{\lambda}) \big)
=
\circ ( r^n ) 
\\
&\E^{\mu}_1 [ \TT_2 ]
=
\TT_1
\,\, 
\text{and} \,\, \Big\Vert S_n(r, \TT)  \Big\Vert_{\mathbb{L}^{p}_{\W_p^{\rm ad}} ( \mu )}^p
=
1
+
\circ ( r^{n-1} )
.\end{split}
\end{equation*}
\item \label{cond: concavity martingale} There exists $\kappa > 0 $ such that
$$
\sup_{ \substack{ T \in \mathbb{L}^{p}_{{\rm ad}} ( \mu ) \\ \Vert T \Vert_{\mathbb{L}^{p}_{{\rm ad}} ( \mu ) } }}
\Big\{
\mathbb{E}^{\mu}\big[ ( T\cdot \partial_{xx} \delta_m g) T + h_1' T_1 ( T_2 - T_1) \big]
+
\mathbb{E}^{\mu^{\otimes 2}} \big[ ( \partial_{ \hat{x} , x} \delta^2_m g ) ( T( X), T( \hat{X} ) \big] 
\Big\}
\leq 
-\kappa 
.$$
\end{enumerate}
\end{AssumptionD1}

\begin{Proposition} \label{prop:mart expansion}
Assume that $g$, $ \mu $, $n$ and $p$ satisfy Assumptions \hyperref[ass:reformulation 1 wasserstein]{${\rm A}_\mathbf{d}$}, \hyperref[ass:reformulation 2 gen dist]{${\rm B}_\mathbf{d}$} and Assumptions \hyperref[ass: existence approx sad-point Mart adapted]{${\rm D}_{\rm M}$}, then $G^{\rm M}_{\rm ad}$ admits the following expansion 
$$
G^{\rm M}_{\rm ad}( r ) 
=
L^{n, r}_{\W_p^{\rm ad}} \big( S_n(r, \TT) \big) 
+
\circ ( r^n ) 
$$
where $\TT$ is defined in Assumptions \hyperref[ass: existence approx sad-point Mart adapted]{${\rm D}_{\rm M}$}.
\end{Proposition}

We now move on to a sufficient condition in order to have the order $2$ expansion of the martingale distributionally robust optimisation under adapted Wasserstein constraints.

\begin{Proposition}\label{prop:mart expansion order 2}
Let $ g(\mu) = \int_{\R^2} g(x) \mu( \mathrm{d}x ) $ and $\mu$, $p$ satisfy Assumptions \hyperref[ass:reformulation 1 adapted wasserstein]{${\rm A}_{ \W^{\rm ad}_p } $}, \hyperref[ass:reformulation 2 gen dist]{${\rm B}_{ \W^{\rm ad}_p }$} for $n =2$. Let $J := (1, -1)$. There exists a unique solution $h_1$ to
$$
\inf_{ h \in \mathbb{L}^{p'} ( \mu )} \Vert \partial_x^{\rm ad} \delta_m g + h J \Vert_{ \mathbb{L}^{p'} ( \mu )}
.$$
Assume that $h_1$ admits a $C^2$ representation such that $ \mu' \mapsto g( \mu') + \int h_1(x_1) (x_2-x_1) \mu' ( \mathrm{d} x ) $ satisfies \hyperref[ass:reformulation 2 gen dist]{${\rm B}_{ \W^{\rm ad}_p }$} and \hyperref[ass:exist approx saddle point adapted wasserstein]{${\rm C}_{ \W^{\rm ad}_p }$}. Finally, define $h_2$ by:
\begin{equation*}
\begin{split}
& T^1 = \frac{ w_{p'} \big( \partial_x^{\rm ad} g + Jh_1 \big) }{\Vert \partial_x^{\rm ad} g + Jh_1 \Vert_{\mathbb{L}^p_{\rm ad} (\mu) }}
\\
&h_2 
=
\frac{
 \sum_{i = 1}^{2} 
 \E^{\mu}_1 [ 
 \frac{(\partial_{x_i x_1} g) T^1_i }{ \vert T^1_1 \vert^{p-2} }
-
 \frac{ (\partial_{ x_i x_2}g) T^1_i}{ \vert T^1_2 \vert^{p-2} }
 ]
-h'_1 \E^{\mu}_1 [ \frac{ 1 }{ \vert T^1_{2} \vert^{p-2} } - \frac{ 1 }{ \vert T^1_1 \vert^{p-2} } ]
}
{
1 + \vert T^1_1 \vert^{p-2} \E^{\mu}_1[ \frac{1}{ \vert T^1_{2} \vert^{p-2} } ]
 }
.\end{split}
\end{equation*}
If $h_2$ admits a $C^1$ representation, then 
$$
G^{\rm ad}_{\rm M} ( r ) = g( \mu ) + r\Vert \partial_x^{\rm ad} \delta_m g - h_1 J\Vert_{ \mathbb{L}^{p'} ( \mu ) } + \frac{r^2}{2} \E^{\mu} \big[ T^1\cdot (\partial^2_{x} \delta_m g ) T^1 \big] + \circ(r^2) 
.$$
\end{Proposition}

\begin{Remark}
{\rm 
The assumption on the regularity of $h_1$ and $h_2$ could be relaxed to $C^1$ and $C^0$ respectively, with additional effort. However, due to the technical aspect of this paper, we did not want to include it for the sake of clarity. From a practical point of view, the regularity of $h_1$ can be verified as $h_1$ is characterised by the first-order condition 
$$
 w_{p'} \big(\E_1^\mu[ \partial_{x_1} f ] - h_1\big)= 
 \E_1^\mu \big[ w_{p'} (\partial_{x_2} f + h_1 ) \big]
.$$
which can be rewritten as an equation of the form $ F(x_1, h_1(x_1)) =0$ for some $F: \R \times \R \rightarrow \R$. Using the implicit function theorem, one can deduce regularity properties of $h_1$.
}
\end{Remark}

\section{Numerical illustration}\label{sec:Numerical illustration}

We will compare the different approximations to see which one is the best compromise between simplicity of computation and precision. For the sake of the example, we assume that $ g $ has the following form: $ g ( \mu ) = \int F \big( y, \mu(f) \big) \mu(\mathrm{d}y) $, where $ F : (x, p) \in \R^2 \times \R \mapsto F(x, p) \in \R$ and $ f : \R^2 \rightarrow \R$. In this case, we have 
\begin{equation*}
\begin{split}
\delta_m g ( \mu, x) &= F \big( x, \mu( f) \big) + f(x) \int \partial_p F \big(y, \mu(f) \big) \mu (\mathrm{d}y) 
\\
\delta^2_m g ( \mu, x, \hat{x}) &= f(\hat{x}) \partial_p F \big( x, \mu( f ) \big) + f(x) \partial_p F \big( \hat{x}, \mu( f ) \big) + f(x) f(\hat{x} )\int \partial^2_{pp} F \big(y, \mu(f) \mu (\mathrm{d}y) 
\end{split}
\end{equation*}

Following Proposition \ref{prop:approx 1}, we want to compare the following quantities,
\begin{equation*}
\begin{split}
&\sup_{ \substack{ T \in \mathbb{L}^p ( \mu ) \\ \Vert T \Vert_{\mathbb{L}^p ( \mu )} \leq 1 } } g( \mu^{ r, T} )
\,\, , \,\,
G^{\rm quad}(r) := \sup_{ \substack{ T \in \mathbb{L}^p ( \mu ) \\ \Vert T \Vert_{\mathbb{L}^p ( \mu )} \leq 1 } } L^{2, r} (T )
\,\, \text{and} \,\, L^{2, r} (T_1 + r T_2 )
\end{split}
\end{equation*}
and to compare the precision/time of computation. Following the approach of \citeauthor{nendel_parametric_2022} \cite{nendel_parametric_2022}, we will solve all optimisation problems using Neural Networks, and we will use the following expression of $ L^{2, r}$.
\begin{align*}
L^{2, r} ( \Theta_1, \Theta_2 )
&= \alpha
 +
 r \sum_{i = 1}^2 \E^{\mu} \Big[
 \partial_{x_i} F \big( X, \mu(f) \big) \Theta_i + \beta \big(\partial_{x_i} f \big) \Theta_i \Big]
 \\&
 + \frac{r^2}{2}
 \sum_{i,j = 1}^2
 \E^{\mu} \Big[ \partial_{x_i x_j} F \big( X, \mu(f) \big) \Theta_i \Theta_j
 + \beta \big(\partial_{x_i x_j} f \big) \Theta_i \Theta_j \Big]
 \\
 &+ \frac{r^2}{2}
 \sum_{i,j = 1}^2 \gamma \E^{\mu} [ \partial_{x_i} f \Theta_i ] \E^{\mu} [ \partial_{x_j} f \Theta_j ]
 + 2 \E^{\mu} [ \partial_{x_i} f \Theta_i ] \E^{\mu} [ \partial_{x_j y} F \Theta_j ]
\\
\alpha &:= \mu \Big( F \big( \cdot, \mu(f) \big) \Big)
\,\,\,
\beta := \mu \Big( \partial_y F \big( \cdot, \mu(f) \big) \Big)
\,\,\,\,\,
\gamma := \mu \Big( \partial^2_{yy} F \big( \cdot, \mu(f) \big) \Big)
\end{align*}
According to Proposition \ref{prop:order 2 expansion gen dist}, in the Wasserstein case,
$$
\left\{
\begin{array}{ll}
 T^1_i(X) := \frac{ w_{p'}( \partial_{x_i}F + \beta \partial_{x_i}f ) }{(\lambda_1 p)^{p'-1} } \,\,\, i = 1,2 \\
 \lambda_1 p := \big( \sum_{i=1}^2 \Vert \partial_{x_i}F + \beta \partial_{x_i}f \Vert_{\mathbb{L}^{p'}(\mu)}^{p'} \big)^{1/p'} \\
 T^2_i(X) := \frac{1}{ p(p-1) \vert T^1_i(X) \vert^{p-2} \lambda_1 }
 \Big(
 ( \partial_{x_i} \nabla_x) F \cdot T^1
 +
 \beta ( \partial_{x_i} \nabla_x)f \cdot T_1
 + \gamma \mu( \nabla_x f \cdot T^1 ) \partial_{x_i} f
 \\
 \hspace{4.5 cm }
 + \mu( \nabla_x f \cdot T^1 ) \partial_{x_i y} F
 + \mu( \nabla_{x} \partial_y F \cdot T^1 ) \partial_{x_i} f
 \\
\hspace{4.5 cm }
  -
 \lambda_2 p w_p( T^1_i )
 \Big)
 \,\,\, i = 1,2 \\
 \lambda_2 p := \mu( T^1 \cdot D^{2}_xF \, T^1 ) + \beta \mu( T^1 \cdot D^{2}_xf \, T^1 )  + \gamma \mu(\nabla_x f \cdot T^1)^2 \\
 \hspace{1.1 cm }
 + \mu(\nabla_x f \cdot T^1) \mu(\nabla_x \partial_y F \cdot T^1)
\end{array}
\right.
.$$

We have the following graphs.

\begin{figure}[H]
\begin{centering}
\includegraphics[scale = 0.45]{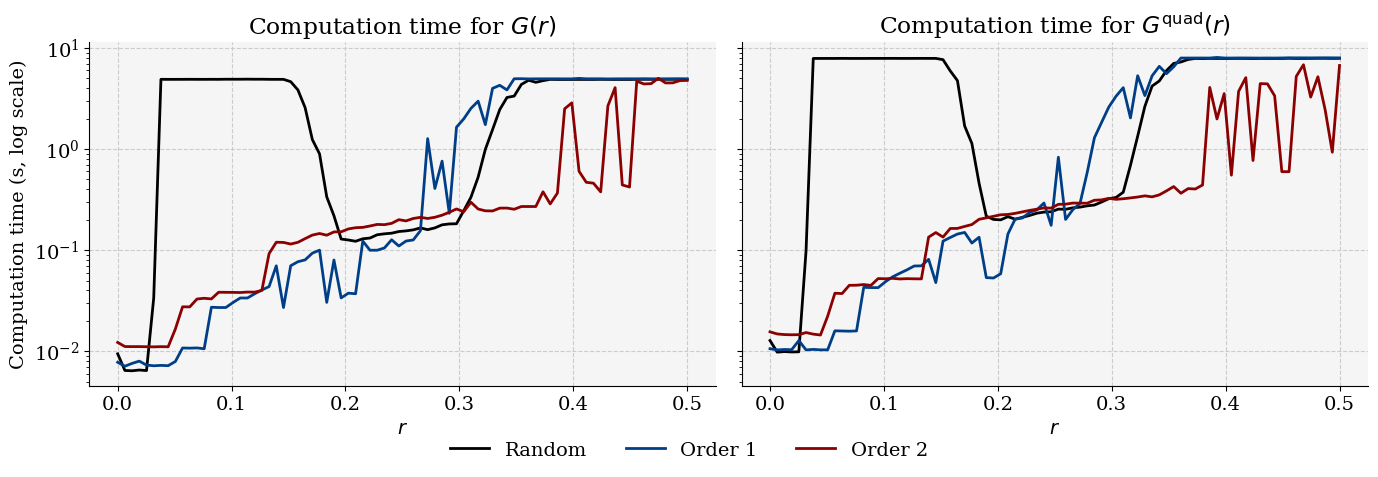}
\caption{\it \footnotesize Time computations.}
\label{fig:comp_time}
\end{centering}
\end{figure}

We observe that the computation times are significantly higher for the random initialisation, even though we averaged the results over $25$ different random seeds. We now turn to the question of whether the second-order approximation is sufficiently accurate.

\begin{figure}[H]
\begin{centering}
\includegraphics[scale = 0.45]{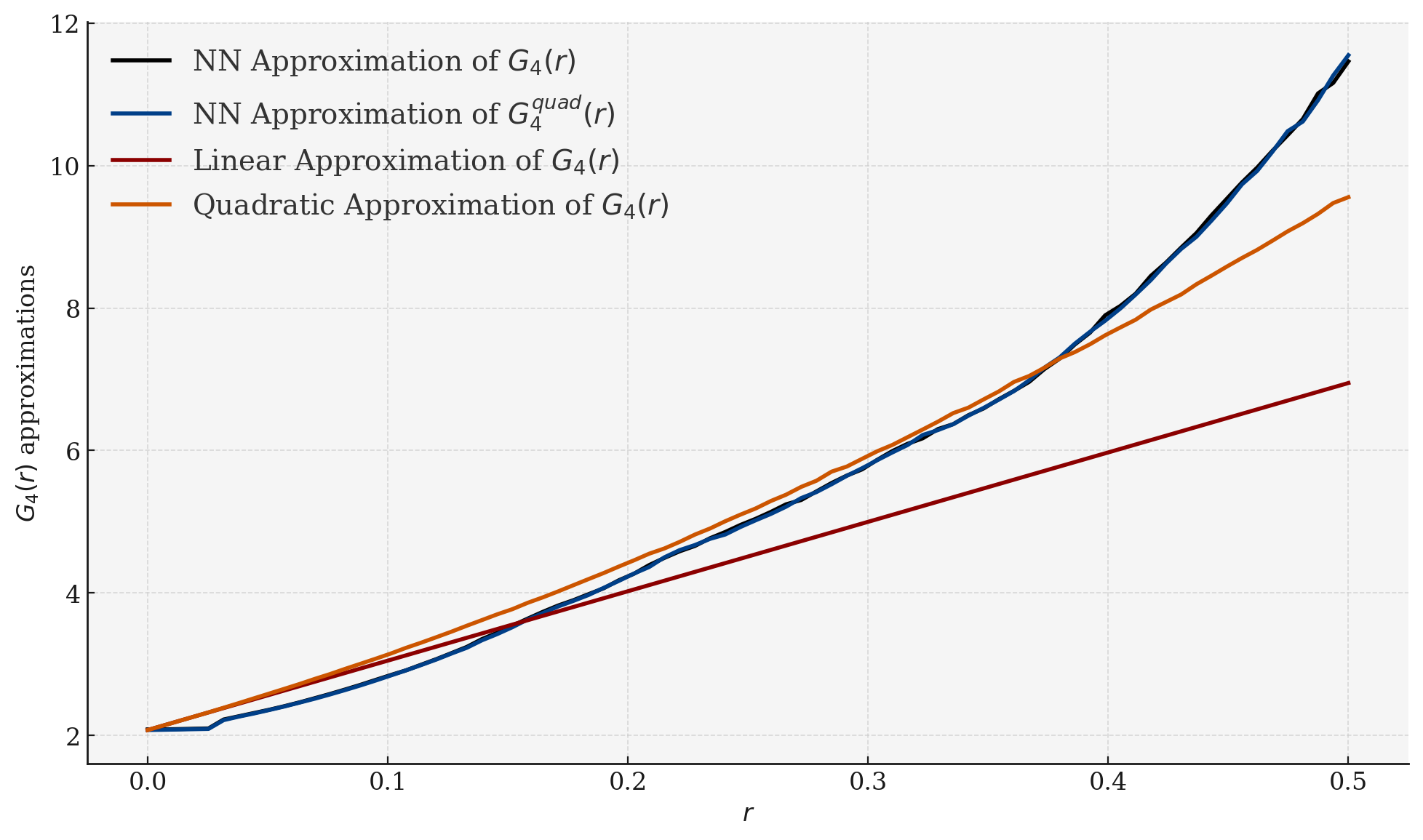}
\caption{\it \footnotesize Quality of approximations.}
\label{fig:quality_approx}
\end{centering}
\end{figure}

\noindent As shown in Figure~\ref{fig:quality_approx}, it is clear that the quadratic approximation greatly enhances the efficiency of solving the DRO problem with neural networks. The linear approximation of $G^{\W_p}(r)$ deviates significantly from the true value once $r > 0.25$. In contrast, $L^{2, r}(T_1 + r T_2)$ already provides a good approximation, and $G^{\rm quad}(r)$ is even more accurate. However, as shown in Figure~\ref{fig:comp_time}, this improved accuracy does not translate into a more favourable computational time for $G^{\rm quad}(r)$.

\section{Proofs}\label{sec:Proofs}
\subsection{Proof of Proposition \ref{prop:approx 1} }

Let $ T : S \rightarrow S$, $ n \geq 1 $ be an integer and $ I \subset \llbracket 1, n \rrbracket := \{ 1 , \cdots, n\}$. Define $ T^{\otimes I } : S^{n} \rightarrow S^{n}$, as $ \big( T^{\otimes I }( x^1, \cdots, x^n) \big)_i:= T(x^i) $ if $ i \in I$ and $ x^i $ otherwise, for $1 \leq i \leq n$. Also, for $ \lambda \in \R^d $ and $ x \in S^d $, we set $ \lambda x := ( \lambda_i x_i )_{1 \leq i \leq d} \in S^d$. Finally, for $T \in \mathbb{L}^p (\mu)$, we define $\mu^{r, T} := \mu \circ ( X + r T)^{-1}$.

\begin{Lemma}\label{lemma: cont thm Lp}

Let $F : \Pc_p ( \R ) \times \R^\ell \times \R^\ell \rightarrow \R $ be a continuous function, with $p/q-$polynomial growth in the space variable $x$, locally uniformly in $ \mu $, for some $q > 1$. Then 
$$ (m, T ) \in \Pc^p ( \R ) \times \mathbb{L}^p ( \mu ) \mapsto F \big( m, T^{\otimes \ell} , X \big) \in \mathbb{L}^q ( \mu^{\otimes \ell} ) 
\text{
 is continuous} 
.$$
\end{Lemma}
\noindent{\textbf{Proof of Lemma \ref{lemma: cont thm Lp} }}. Let $T_n \rightarrow T$ in $\mathbb{L}^p ( \mu )$, and $ m_n \rightarrow \mu $ in $ \Pc^p ( \R )$. Since $ F $ is continuous, we have
$
F \big( m_n , T^{\otimes \ell}_n , X \big) 
\xrightarrow[ n \rightarrow \infty]{\P}
F \big( m, T^{\otimes \ell}, X \big) 
.$
 Also, by the growth assumption on $F$, since $ m_n $ converges toward $m$, there exists $ C>0$ such that for sufficiently large $n$, $
 \vert F \big( m_n , T^{\otimes \ell}_n , X \big) 
 \vert^q
 \leq 
 \hat C \big( 1 + \vert X \vert^p + \vert T^{\otimes \ell}_n \vert^p \big) 
.$
 Hence, $ \vert F \big( m_n , T^{\otimes \ell}_n ( X ) , X \big) \vert^q$ is dominated by a uniformly integrable family of random variables and is therefore uniformly integrable.
\ep

\noindent{}The following is a straightforward corollary.
\begin{Corollary}\label{corr:cont map thm measure}
Take $F : \Pc_p ( \R ) \times \R^\ell \rightarrow \R $ to be a continuous function with $p/q-$polynomial growth in the space variable, locally uniformly in $ \mu $. Then for all $\mu_r \rightarrow \mu $ in $ \Pc_p ( \R ) $
$$
\sup_{ \substack{ T \in \mathbb{L}^p ( \mu ) \\ \Vert T \Vert_{\mathbb{L}^p (\mu )} \leq 1 }} 
\Big\Vert 
\Delta_{ r T^{\otimes \ell} } \big[F ( \mu_{r} , \cdot ) \big] 
 \Big\Vert_{ \mathbb{L}^{q} ( \mu^{\otimes \ell} ) }
\xrightarrow[r \rightarrow 0]{} 0 
.$$
\end{Corollary}

\begin{Lemma}\label{Lemma: expansion wrt graph}
Let $ F : \R^\ell \rightarrow \R$ be a $C^n$ function, such that for $ 0 \leq k \leq n$ , $ i \in \mathbb{N}^\ell$ with $ \vert i \vert_1 = k $, $\partial^{\vert i \vert_1}_{\mathbf{x}^i} F$ has $ p/p_{i} = \frac{p}{p - \vert i \vert_1} ' -$ polynomial growth. Let $ T \in \mathbb{L}^p ( \mu )$. Define
$$
f ( r )
:=
\int_{ \R^{\ell}} F ( x ) ( \mu^{r, T} - \mu )^{\otimes \ell} ( \mathrm{d}x ) 
\,\, \text{for $r \in \R$}
.$$
The function $f$ is $C^n $ in a neighbourhood of $0$, and denote by $f^{(i)}$ the $i-th$ derivative for $ 1 \leq i \leq n$. Then, $f$ satisfies:
\begin{enumerate}[label=\textnormal{(\roman*)}]
\item \label{eq:repres int graphe}
$
f ( r ) =
r^\ell
\int_{ (0, 1 )^\ell} 
\mathbb{E}^{\mu^{\otimes \ell }} 
			\big[
\partial^{\ell}_{ \mathbf{x}^{ \llbracket 1, \ell \rrbracket} } F ( X + r \lambda T^{\otimes \ell } )
\prod_{a=1}^\ell T \left( X^a \right)
 \big]
\mathrm{d}\lambda
$
\item \label{eq:diff at 0 int graph }
For $ 0\leq s \leq n $,
$
f^{(s)} ( 0 ) 
=
\left\{
 \begin{array}{ll}
 0 & \mbox{if } s \leq \ell-1 \\
 \sum_{ \substack{ i \in \mathbb{N}_*^\ell \\ \vert i \vert_1 = s}}
\frac{ s!
}{
i!
}
\mathbb{E}^{ \mu^{\otimes \ell}} \Big[
\big( \partial^{ s }_{ \mathbf{x}^i } F \big) \prod_{a=1}^{\ell} T ( X^a )^{i_a} \Big] & \mbox{otherwise.}
 \end{array}
\right.
.$
\item \label{eq:estim at 0 int graph}
$
\Big\vert 
f( r ) 
-
\sum_{s = \ell }^{n} 
\frac{f^{(s)} ( 0 ) }{s ! }
r^s
\Big\vert 
\leq 
r^n 
C_{n, \ell} 
\Vert T \Vert_{\mathbb{L}^p( \mu ) }^\ell
\sup_{ \substack{ I \subset \llbracket 1, \ell \rrbracket 
\\ i \in \mathbb{N}^{ \vert I \vert } \\ \vert i \vert_1 = n }}
\int_0^1
\Big\Vert 
\Delta_{ r t T^{ \otimes I}} \big[
\partial^{n}_{\mathbf{x}^i } F \big] 
\Big\Vert_{ \mathbb{L}^{p_i'}( \mu^{\otimes \ell} )}
\mathrm{d}t
.$
\end{enumerate}
\end{Lemma}

\noindent{}\textbf{Proof of Lemma \ref{Lemma: expansion wrt graph}.} The proof consists of $5$ steps.

\noindent{\bf{Step (i)}:} we first prove by induction that 
\begin{equation}\label{eq:gen formula deriv 0 int graphe}
f ( r )
\! = \!\! \!
\sum_{ J \subset \llbracket 1, \ell \rrbracket } 
\! \!
\E^{\mu^{\otimes \ell}}
\big[
( - 1 )^{ \ell - \vert J \vert_1 }
\!
F ( X \! + \! r T^{\otimes J} )
\big] 
\text{
where} 
\,
\big( x + r T^{\otimes J} ( x ) \big)_j \!= 
  x^j \!+ \! \mathds{1}_{J}(j) r T ( x^j )
.\end{equation}
\noindent{}\underline{Initialisation, $\ell = 1$.} By definition, for $ \ell = 1$, 
$
f ( r ) 
=
\int_{\R} \Delta_{ r T (x ) }[F] (x) \mu ( \mathrm{d}x ) 
$
which is exactly the desired formula for $ \ell = 1 $.
\noindent{}\underline{Assume the formula is proven for some $\ell - 1 \geq 0$.} The integration with the last variable yields, 
$$
\int_{ \R^{\ell}} F (x ) ( \mu^{r, T} - \mu )^{\otimes \ell } ( \mathrm{d}x ) 
=
\int_{\R}
\big( 
\int_{ \R^{\ell-1}} F ( x, x^\ell ) ( \mu^{r, T} - \mu )^{\otimes \ell-1} ( \mathrm{d}x ) 
\big)
( \mu^{r, T} - \mu ) ( \mathrm{d}x^\ell ) 
.
$$
So, by induction hypothesis,
\begin{align*}
\int_{ \R^{\ell}} F ( x ) ( \mu^{r, T} - \mu ) ( \mathrm{d}x ) =
&\int_{\R} 
\sum_{ I \subset \llbracket 1, \ell-1 \rrbracket } 
\!\! ( - 1 )^{ ( \ell - 1 ) - \vert I \vert }
\E^{\mu^{\otimes \ell - 1 }}
\big[
F ( X + r T^{\otimes I} , x^\ell )
\big]
( \mu^{r, T} - \mu ) ( \mathrm{d}x^\ell ) 
\\
&=
\int_\R \sum_{ I \subset \llbracket 1, \ell-1 \rrbracket } \!\!\!\!\!
\E^{\mu^{\otimes \ell - 1}} \Big[
( - 1 )^{ ( \ell - 1 ) - \vert I \vert }
F \big( X \!\!+\! r T^{\otimes I} , x^\ell \!+\! r T ( x^\ell ) \big)
\Big]
\mu ( \mathrm{d}x)
\\
&-
\int_\R
\sum_{ I \subset \llbracket 1, \ell-1 \rrbracket } 
\E^{\mu^{\otimes \ell - 1}} \big[
( - 1 )^{ ( \ell - 1 ) - \vert I \vert }
F ( X + r T^{\otimes I} , x^\ell )
\big]
\mu ( \mathrm{d}x^\ell )
\\
=
&\sum_{ I \subset \llbracket 1, \ell \rrbracket } 
\E^{\mu^{\otimes \ell}}
\Big[
( - 1 )^{ \ell - \vert I \vert }
F \big( X + r T^{\otimes I} \big)
\Big].
\end{align*}

\noindent{\bf{Step (ii)}:} we now prove that $f$ is $C^n$ in a neighbourhood of $0$. We just need to prove that each term of the sum \eqref{eq:gen formula deriv 0 int graphe} is $C^n$. For the sake of clarity, we only prove the case $I = \llbracket 1, \ell \rrbracket $, all other cases are dealt with similarly. If $x$ is fixed, the mapping $ r \mapsto v ( r, x ) := F \big( x + r T^{ \otimes I } ( x ) \big)$ is $C^n$ because of the assumption on $F$. Furthermore for $ 0 \leq s \leq n $ and $x \in \R^\ell$ we have 
$
\partial^{s}_{r^s } v ( r, x ) 
=
\sum_{ \substack{ J \in \mathbb{N}^\ell \\ 
\vert J \vert_1 = s}}
\frac{s!}{J!}
\partial^{s}_{x^J} F \big( x + r T^{\otimes I} ( x ) \big) \prod_{i=1}^\ell T ( x^i ) ^{J_i} 
.$
Hence by $ p / p_{J}' $ polynomial growth assumption on $F$, there exists $ C > 0$ such that 
$$
\vert 
\partial^{s}_{r^s } F ( r, x ) 
\vert 
\leq 
C 
\sum_{ \substack{  J \in \mathbb{N}^\ell \\ 
\vert J \vert_1 = s}}
\big( 1 + \vert x \vert^{p/{p_{J}'}} + r^{p/p_{J}'} \vert T^{ \otimes \ell} ( x ) \vert^{p/p_{J}'} \big) 
\prod_{i=1}^\ell \vert T ( x^i )^{J_i} \vert 
.$$
By Hölder's inequality, we see that the right hand side can be dominated uniformly in $r $ as soon as $r $ is in some compact subset. So, by the differentiation under the integral theorem and the assumptions on $F$, it is clear that $f$ is $C^n$.

\noindent{\bf{Step (iii)}:} we now prove that the equation \ref{eq:repres int graphe} is satisfied. First, notice that 
\begin{align*}
f ( r ) 
&=
r
\int_{ \R ^{\ell-1}}
\E^\mu \Big[
F \big( X+ r T , y \big) 
-
F ( X, y ) 
\Big]
( \mu^{r, T} - \mu )^{\otimes \ell - 1} ( \mathrm{d}y ) 
\\
&=
r
\int_{ \R ^{\ell-1}}
\int_{ (0, 1 )} 
\E^\mu
\big[
\partial_{x^1} F \big( X + \lambda_1 r T , y \big) T 
\big]
\mathrm{d}\lambda_1
( \mu^{r , T} - \mu )^{\otimes \ell - 1} ( \mathrm{d} y ) 
\end{align*}
Then \ref{eq:repres int graphe} follows by direct repetition of this argument.

\noindent{\bf{Step (iv):}} we now prove the point \ref{eq:diff at 0 int graph }. To do so, we proceed by induction on $\ell$.

\noindent{}\underline{Initialisation, for $\ell = 1$.} In this case, 
$
f ( r ) 
=
\E^\mu \big[ F ( X + r T ) - F ( X ) \big]
$
hence 
$$
f^{(s)} ( 0 ) 
=
\mathds{1}_{ \{s \geq 1 \} }
\mathbb{E}^{\mu} \big[
\big( \partial^{ \vert i_s \vert_1 }_{ \mathbf{x}^{i_s} } F \big) T(X)^s \big] \,\, \text{where $i_s = (s, 0 , \cdots, 0 ) \in \N^\ell$}
$$
which is exactly the equation \ref{eq:diff at 0 int graph } for $\ell = 1$. 

\noindent{}\underline{Induction, assume that \ref{eq:diff at 0 int graph } holds for some $\ell \geq 1$.} Set 
$$
U ( x^1, r ) 
:=
\int_{ \R^{\ell-1} } 
F ( x^1, x ) ( \mu^{r, T } - \mu ) ( \mathrm{d}x ) 
,$$
by definition of $f$, 
$
f ( r ) 
=
\E^\mu \big[ U ( X + r T , r ) \big]
-
\E^\mu \big[ U ( X , r ) \big]
.$
Of course, with the assumption on $F$, $x \mapsto U ( x, r ) $ satisfies the assumption of Lemma \ref{Lemma: expansion wrt graph} so, by induction hypothesis, for $ 0 \leq s \leq \ell-2$, 
$$
f^{ ( s ) } ( 0 ) 
=
\sum_{ k = 0}^s
\binom{s}{k}
\E^\mu \Big[
\partial^{s - k}_{ \mathbf{x}^{i_{ s - k } }} \partial^{k}_{ r^k } 
U 
( X, 0 ) 
\Big]
-
\E^\mu \Big[
\partial^s_{r^s} 
U 
( X, 0 ) 
 \Big] 
$$
 where $\partial^s_{r^s}$ is the partial derivative $\partial_r$ applied $s$ times. By the induction hypothesis \ref{eq:diff at 0 int graph }, $\partial^s_{r^s} 
U 
( x^1, 0 ) = 0$ for $ 0 \leq s \leq \ell-2$. So $f^{ ( s ) } ( 0 ) = 0$. 
Now, for $ s = \ell -1$, we get 
\begin{align*}
f^{ ( \ell -1 ) } ( 0 ) 
&=
\sum_{ k = 0}^{\ell-1}
\binom{\ell -1}{k}
\E^\mu
\big[
\partial^{ \ell-1 - k }_{ \mathbf{x}^{ i_{ \ell-1 - k } }} \partial^k_{r^k}  
U 
( X, 0 ) 
\big]
-
\E^\mu \big[
\partial^{\ell - 1}_{r^{\ell-1}} 
U 
( X, 0 ) 
\big]
\\
&=
\E^\mu \big[
\partial^{\ell-1}_{ r^{\ell-1}} 
U 
( X, 0 ) 
 \big]
-
\E^\mu \big[
\partial^{\ell-1}_{r^{\ell-1}} 
U 
( X, 0 ) 
 \big]
=
0
.
\end{align*}
Now, for $ \ell \leq s \leq n $, again using the induction hypothesis \ref{eq:diff at 0 int graph }
\begin{align*}
f^{ ( s ) } ( 0 ) 
&=
\sum_{ k = \ell-1}^{s}
\binom{s}{k}
\E^\mu \big[
\partial^{s-k}_{x^{s - k}} \partial^{k}_{r^k} 
U
( X, 0 )
 T ( X )^{s - k}
 \big] 
-
\E^\mu \big[
\partial^s_{r^{s}} 
U 
(X, 0 ) 
\big]
\\
&=
\sum_{ k = \ell-1}^{s-1}
\binom{s}{k}
\E^\mu \big[
\partial^{s-k}_{x^{s - k}} \partial^{k}_{r^k} 
U
( X, 0 ) 
 T( X )^{ s - k} 
 \big]
\\
&=
\sum_{ k = \ell-1}^{s-1}
\binom{s}{k}
\sum_{ \substack{ i \in \N_*^{\ell-1} \\ \vert i \vert_1 = k} }
\frac{k!}
{ i! }
\E^{\mu^{\otimes \ell}} \Big[ \partial^{s-k}_{ (x^1)^{s-k}} \partial^{k}_{ \mathbf{x}^i} F ( X^1, X ) T( X^1 )^{s - k } \prod_{j=2}^\ell T ( X^{j} ) ^{i_j }
 \Big].
\end{align*}
Now, on the last line, we see that setting $ i_1 = s-k $ we exactly get \ref{eq:diff at 0 int graph }. 

\noindent{\bf{Step (v):}} we now prove the estimate \ref{eq:estim at 0 int graph}. Combining equation \eqref{eq:gen formula deriv 0 int graphe}, and Leibniz formula, we get 
$$
f^{ (n)} ( r )
=
\sum_{ I \subset \llbracket 1, \ell \rrbracket } 
\sum_{ \substack{ i \in \mathbb{N}^{ \vert I \vert } \\ \vert i \vert_1 = n }}
\frac{ ( - 1 )^{ \ell - \vert I \vert } n ! 
}{i !}
\E^{\mu^{\otimes \ell}} \big[
\partial^{n}_{\mathbf{x}^i_I} F ( X + r T^{\otimes I} ) \prod_{ \substack{ i \in I \\ 1 \leq j \leq \vert I \vert }} T ( x^i )^{ i_j } 
 \big]
.$$
Hence, by Taylor's formula
$$
f ( r ) \!
-\!\!
\sum_{k = \ell }^{n} \!\!
\frac{f^{(k)} ( 0 ) }{k ! }
r^k \!
=\!
r^n  \!\!\!
\int_0^1 \!\!\!\!
n ( 1 \!-\! t )^{n-1} 
\!\!\!\!
\sum_{ I \subset \llbracket 1, \ell \rrbracket } \!
\sum_{ \substack{  i \in \mathbb{N}^{ \vert I \vert } \\ \vert  i \vert_1 = n }}
\!\!\!
\frac{ ( - 1 )^{ \ell - \vert I \vert }
}{ i !}
\E^{\mu^{\otimes \ell}} \Big[ 
\Delta_{r t T^{\otimes I} } \big[
\partial^{n}_{\mathbf{x}^i_I} F \big]
\! \!\!\!\!\!
\prod_{ \substack{ a \in I \\ 1 \leq j \leq \vert I \vert }} \!\!\!\!\!
T ( X^a )^{ i_j } \!
\Big]
\mathrm{d}t
$$
So, by Hölder's inequality and since
$$
\mathbb{E}^{ \mu ^{\otimes \ell} } 
\big[ 
\prod_{a \in I } 
\vert 
T ( X^a ) 
\vert^{i_a p_i}
\big]^{1/p_i} 
=
\prod_{ a \in I}
\mathbb{E}^\mu
\big[ 
\vert 
T ( X ) 
\vert^{i_a p_i}
\big]^{1/p_i},
$$
we get
\begin{align*}
\Big\vert 
f ( r ) 
\!-\!
\sum_{k = \ell }^{n} 
\frac{f^{(k)} ( 0 ) }{k ! } r^k\!
\Big\vert
\!
\leq 
r^n 
C_{n,\ell}
\Vert T \Vert_{\mathbb{L}^p (\mu)}^\ell
\int_0^1
\!\!\!
\sup_{ I \subset \llbracket 1, \ell \rrbracket } 
\sup_{ \substack{  i \in \mathbb{N}^{ \vert I \vert } \\ \vert i \vert_1 = n } }
\mathbb{E}^{ \mu^{\otimes \ell } }
\Big[ 
\vert 
\Delta_{ r t T^{\otimes I} ( X )} \big[
\partial^{n}_{\mathbf{x}^i_I} F \big] ( X ) 
\vert^{p_i'}
\Big]^{1/p_i'} 
\mathrm{d}t
\end{align*}

\ep

\noindent \textbf{Proof of Proposition \ref{prop:approx 1}}. We distinguish the settings $\mathbf{d} = \W_p $ and $ \W_p^{\rm ad}$.

\noindent{\underline{$1.$ Wasserstein ball DRO setting.}} We first prove that equation \eqref{eq:first reform gen dist} holds. Let $
\Pi ( \mu ) 
:=
\lbrace 
\pi \in \mathcal{P} ( \R \times \R ) 
\,\,\, 
\pi \circ X^{-1} = \mu 
\rbrace
$. We can rewrite 
$$
G^{\W_p} ( r ) 
=
\sup_{ \pi \in \Pi ( \mu )} \phi ( \pi ) 
\,\, \text{ where } \phi (\pi) := g( \pi \circ \left. X' \right.^{-1} ) + \chi_{D_p (\mu, r) } ( \pi ) 
$$
and $ \chi_{D_p (\mu, r) } $ is the characteristic function of $D_p (\mu, r)$ defined by \begin{eqnarray*}
D_p (\mu, r) 
:=
\big\{\pi\in\Pc_p( \X ):~\pi\circ X^{-1}=\mu
  ~\mbox{and}~
  \E^\pi|X-X'|^p\le r^p
\big\}.
\end{eqnarray*}
It is clear that $D_p (\mu, r)$ is closed with respect to the weak topology, and $ g $ is continuous with respect to the weak topology; hence, $ \phi $ is lower semi-continuous with respect to the weak topology. Furthermore, by Assumption \hyperref[ass:reformulation 1 wasserstein]{${\rm A}_{\W_p}$}, $\mu$ is atomless, so by Lemma $2.1$ of \cite{beiglbock2018denseness}, the set 
$
\Pi_0 ( \mu ) 
:=
\lbrace 
\mu ( \mathrm{d}x ) \delta_{\phi ( x ) } ( \mathrm{d}x' ) 
,
\phi : \R \rightarrow \R \,\,\, \text{ is measurable} 
\rbrace
$
is dense in
$
\Pi ( \mu ) $ for the weak topology. This shows that the maximisation over $\Pi(\mu)$ reduces to the optimisation over Monge transport plans $T$ that satisfy the appropriate integrability condition, which after a change of variable exactly gives equation \eqref{eq:first reform gen dist}. 

\noindent We now prove that equation \eqref{eq:formulation 2 gen dist} holds. By applying the Taylor expansion for linear functional derivative of Proposition \ref{prop:Taylor expansion measure order n}, we have for $ T \in \mathbb{L}^p (\mu)$, 
$$
 g(\mu^{r, T})
 =
 P_{\mu} ( \mu^{r, T} ) 
+
R_n ( \mu^{r, T},\mu ) 
$$
where $ P_{\mu} ( \mu^{r, T} ) 
 =
 g(\mu)
+
\sum_{k = 1}^n 
\frac{1}{k!}
\int_{\R^k}
\delta^{k}_m g( \mu, x )(\mu^{r, T}-\mu)^{\otimes k} ( \mathrm{d } x )$ and 
$$
R_n ( \mu^{r, T},\mu ) =
\int_0^1 \frac{ ( 1 - t )^{n-1 }}{ ( n - 1 ) !} 
\int_{\R^{n}}
\big( 
\delta^{n}_m g( \bar\mu_t , x )
-
\delta^{n}_m g( \mu , x )
\big)
(\mu^{r, T} -\mu)^{ \otimes n} ( \mathrm{d} x )
\mathrm{d}t 
.$$
Hence, we have 
$$
\Big\vert 
\sup_{ \substack{ T \in \mathbb{L}^p ( \mu ) \\ \Vert T \Vert_{\mathbb{L}^p ( \mu ) } \leq 1 } }
g ( \mu^{r, T} ) 
-
\sup_{\substack{ T \in \mathbb{L}^p ( \mu ) \\ \Vert T \Vert_{\mathbb{L}^p ( \mu) } \leq 1 } } 
 P_{\mu} ( \mu^{r, T} ) 
\Big\vert 
\leq 
I_n ( r ) 
\,\,
\text{where } 
I_n ( r ) 
:= \! \!
\sup_{\substack{ T \in \mathbb{L}^p ( \mu ) \\ \Vert T \Vert_{\mathbb{L}^p ( \mu ) } \leq 1 } } 
\Big\vert 
R_n ( \mu^{r, T},\mu )
\Big\vert
.$$
Using equation \ref{eq:repres int graphe} of Lemma \ref{Lemma: expansion wrt graph} we know that for $ 0 \leq t \leq 1 $,
\begin{align*}
&\int_{\R^{n}}
\big( 
\delta^{n}_m g(\bar\mu_t, x )
-
\delta^{n}_m g( \mu , x )
)
(\mu^{r, T}-\mu)^{ \otimes n} ( \mathrm{d} x)
\\
&=
r^n
\int_{ (0, 1 )^n} 
\mathbb{E}^{\mu^{\otimes n}} \Big[
\Delta_{ \lambda T^{\otimes n } } \big[ F( \bar{\mu}_t , \cdot ) - F( \mu , \cdot ) 
\big]
( X ) 
\prod_{i=1}^n T ( X^i )
 \Big]
\mathrm{d}\lambda
\end{align*}
where $  F( \mu , x ) = 
\partial^{n}_{ \mathbf{x}^{ \llbracket 1 , n \rrbracket } } \delta^{n}_m g ( \mu , x )$ and $ \bar{\mu} _t  := (1-t) \mu^{r,T} + t \mu $. By Hölder's inequality and Fubini's theorem, we get 
\begin{align*}
I_n ( r ) 
&\leq 
r^n 
\!\!\!
\sup_{\substack{ T \in \mathbb{L}^p ( \mu ) \\ \Vert T \Vert_{\mathbb{L}^p ( \mu ) } \leq 1 } } 
\int_0^1 
\int_{ ( 0, 1 )^n }
\frac{ ( 1 - t )^{n-1 }}{ ( n - 1 ) !} 
\big\Vert
\Delta_{ r \lambda T^{\otimes n } } \big[ 
F( \bar{\mu}^{r, T}_t , \cdot) 
-
F ( \mu , \cdot ) 
\big](X)
\big\Vert_{\mathbb{L}^{p'} ( \mu^{\otimes n} )}
\mathrm{d}\lambda
\mathrm{d}t
.\end{align*} 
By Corollary \ref{corr:cont map thm measure}, for $ 0\leq t \leq 1 $ and $ \lambda \in ( 0, 1 )^n $ 
$$
\sup_{ \substack{ T \in \mathbb{L}^p ( \mu ) \\ \Vert T \Vert_{\mathbb{L}^p ( \mu ) } \leq 1 } } 
\big\Vert
\Delta_{ r \lambda T^{\otimes n } } \big[ 
F( \bar{\mu}^{r, T}_t , \cdot ) 
-
F ( \mu , \cdot ) 
\big]
\big\Vert_{\mathbb{L}^{p'} ( \mu^{\otimes n} )}
\xrightarrow[r \rightarrow 0]{} 0 
.$$
Furthermore, by Assumption \hyperref[ass:reformulation 2 gen dist]{${\rm B}_{\W_p}$} \ref{cond:estimate diff gen dist} on $g$, we can bound the left-hand side uniformly in $r$, proving that $ I_n(r) \rightarrow 0 $ by dominated convergence. Hence, we proved that
$$
\sup_{\substack{ T \in \mathbb{L}^p ( \mu ) \\ \Vert T \Vert_{\mathbb{L}^p ( \mu ) } \leq 1 } } 
g ( \mu^{r, T} ) 
=
g(\mu)
+
\sup_{ \substack{ T \in \mathbb{L}^p ( \mu ) \\ \Vert T \Vert_{\mathbb{L}^p ( \mu ) } \leq 1 } } 
\sum_{\ell = 1}^n 
\frac{1}{\ell!}
\int_{ \R^\ell }
\!\!\! 
\delta^{\ell}_m g( \mu , x )(\mu^{r, T}-\mu)^{\otimes \ell} (dx) 
+
\circ (r^n ) 
.$$
Now, set 
$
f_{\ell, T} ( r ) 
:=
\int_{ \R^\ell }
\delta^{\ell}_m g( \mu , x )(\mu^{r, T}-\mu)^{\otimes \ell} (dx) 
$, by formula \ref{eq:diff at 0 int graph } of Lemma \ref{Lemma: expansion wrt graph} and Taylor's formula 
$$
f_{\ell, T} ( r ) 
=
\hat{\Lc}^{r, \ell}_{\W_p} (T)  
+
r^n
R_{n, \ell, T } ( r)
$$
where $\hat{\Lc}^{r, \ell}_{\W_p}$ is defined by equation \eqref{eqdef:operator polyn} and $ R_{n, \ell, T } ( r ) := \int_{0}^1
\frac{ ( 1 - t )^{n-1} }{ ( n - 1 )! } 
\big( f^{ ( n ) }_{\ell, T} ( r t ) - f^{ ( n ) }_{\ell, T} ( 0 ) \big)
\mathrm{d}t$. A direct consequence of this equality is 
\begin{align*}
\Big\vert
g( \mu ) + \!\!\!
\sup_{ \substack{ T \in \mathbb{L}^p ( \mu ) \\ \Vert T \Vert_{\mathbb{L}^p ( \mu ) } \leq 1 } } 
\sum_{\ell = 1}^n 
\frac{1}{\ell!}
f_{\ell, T} ( r ) 
-
\!\!\!
\sup_{ \substack{ T \in \mathbb{L}^p ( \mu ) \\ \Vert T \Vert_{\mathbb{L}^p ( \mu ) } \leq 1 } } 
\!\!
L^{n, r} ( T ) 
\Big\vert 
\leq 
r^n 
\sum_{\ell = 1}^n 
\sup_{ \substack{ T \in \mathbb{L}^p ( \mu ) \\ \Vert T \Vert_{\mathbb{L}^p ( \mu ) } \leq 1 } } 
\vert R_{n, \ell ,T} ( r ) \vert
\end{align*}
Now, using the estimate \ref{eq:estim at 0 int graph} of Lemma \ref{Lemma: expansion wrt graph}, we have 
$$
\sum_{\ell = 1}^n 
\sup_{ \substack{ T \in \mathbb{L}^p ( \mu ) \\ \Vert T \Vert_{\mathbb{L}^p ( \mu ) } \leq 1 } } 
\vert R_{n, \ell ,T} ( r ) \vert
\leq 
C_n
\sup_{ \substack{ 1 \leq \ell \leq n , \,\, I \subset \llbracket 1, \ell \rrbracket \\ 
 i \in \mathbb{N}^{ \vert I \vert } , \,\vert  i \vert_1 = n } }
\sup_{\substack{ T \in \mathbb{L}^p ( \mu ) \\ \Vert T \Vert_{\mathbb{L}^p ( \mu ) } \leq 1 } } 
\int_0^1 
\Big\Vert 
\Delta_{ r t T^{\otimes I}} \big[
\partial^{n}_{\mathbf{x}^{ i} } \delta_m^\ell g ( \mu, \cdot ) \big]
\Big\Vert_{ \mathbb{L}^{p_i'} ( \mu^{\otimes \ell } )}
\mathrm{d}t
.
$$
Again, the right hand side goes to $0$ by Lemma \ref{lemma: cont thm Lp} and dominated convergence theorem.

\noindent{\underline{$2.$ The adapted Wasserstein ball DRO setting.}} 
The proof follows the same structure. The only difference is that we need to be careful with bi-causal couplings and causal couplings. But, by Lemma $3.1$ of \cite{bartlsensitivityadapted}, bi-causal couplings are dense in the set of causal couplings for the weak topology. Furthermore, by Theorem $3.1$ of \cite{beiglbock2018denseness}, causal Monge couplings are dense in the set of causal couplings. For the proof of the second point, the difference is in the expression, which follows from a straightforward extension of Lemma \ref{Lemma: expansion wrt graph} to the two-dimensional case. Let $ T \in \mathbb{L}^p_{\rm ad} ( \mu )$, $F : ( \R^2 )^{\ell} \rightarrow \R $ a $C^n$ function and $ f ( r ) := \int_{ ( \R^2 )^{\ell} } F( x_1, x_2 ) ( \mu^{r, T} - \mu ) ( \mathrm{d}x_1, \mathrm{d}x_2 )$. Then, letting $\mathbf{1}^\ell = ( 1, \cdots, 1 ) \in \N^\ell$, we have 
$$
f ( r ) =
r^\ell
\sum_{ \substack{ j \in \N^\ell \\ j \leq \mathbf{1}^\ell} }
\int_{ (0, 1 )^\ell} 
\mathbb{E}^{ \mu^{\otimes \ell} }
			\Big[
\partial^{\ell}_{ \mathbf{x}_1^j \mathbf{x}_2^{\mathbf{1}^\ell - j } } F ( X + r \lambda T^{\otimes \ell} )
\prod_{a=1}^\ell T_1 ( X^a_1 )^{j_a} T_2 ( X^a )^{1 - j_a} 
 \Big]
\mathrm{d}\lambda
$$
with 
$
X + r \lambda T^{\otimes n } ( X )
:=
\big( X^1 + r \lambda_1 T ( X^1 ) , \cdots , X^\ell + r \lambda_\ell T ( X^\ell ) \big) 
.$
For $ 0\leq s \leq n $,
$$
f^{(s)} ( 0 ) 
=
\mathds{1}_{\{ s \geq \ell \}}
 \sum_{ \substack{ i \in \mathbb{N}_*^\ell \\ \vert i \vert_1 = s} }
 \sum_{ \substack{ j \in \N^\ell \\ j \leq i} }
\frac{ s!
}{
i!
}
\mathbb{E}^{\mu^{\otimes \ell}} \Big[
\partial^{\ell}_{ \mathbf{x}_1^j \mathbf{x}_2^{i - j } } F ( X )
\prod_{a=1}^\ell T_1 ( X^a_1 )^{j_a} T_2 ( X^a )^{i_a - j_a} 
 \Big] 
.$$
We have the following estimate
$$
\Big\vert 
f( r ) 
-
\sum_{s = \ell }^{n} 
\frac{ f^{ ( s )} ( 0 ) }{s ! }
r^s 
\Big\vert 
\leq 
r^n 
C_{n, \ell} 
\Vert T \Vert_{\mathbb{L}^p (\mu )}^\ell
\sup_{ I \subset \llbracket 1, \ell \rrbracket} 
\sup_{\substack{  i \in \mathbb{N}^{ \vert I \vert } \\ \vert i \vert_1 = n }}
\int_0^1
\Vert 
\Delta_{ r t T ( X ) } [
\partial^{n}_{\mathbf{x}^i } F ] ( X ) 
\Vert_{ \mathbb{L}^{p_i'} ( \mu^{\otimes \ell} )}
\mathrm{d}t
.$$
\ep
\subsection{Proof of Proposition \ref{prop:expansion order n} }

In this subsection, given a map $F : \mathbb{L}^p (\mu) \rightarrow \R$ for $p > 1 $, that is Fréchet differentiable, we denote by $ \nabla F $ the unique element in $\mathbb{L}^{p'} (\mu)$ such that for $T, h \in \mathbb{L}^p (\mu)$, $F(T +h ) = F(T) + \E^{\mu} [ \nabla F (T) h ]+ \circ_{\mathbb{L}^p (\mu)} (h) $. In the following, for $\mathbf{d}$ being either $ \W_p $ or $\W_p^{\rm ad}$, and $ T \in \mathbb{L}^p_{\mathbf{d}} (\mu) $ and $ \lambda \in \R$, 
\begin{equation}\label{eqdef:lagrangien} 
\mathcal{L}_{\mathbf{d}}^{n ,r} (T, \lambda) 
=
L_{\mathbf{d}}^{n ,r} (T) - \lambda \Vert T \Vert_{\mathbb{L}^p_{\mathbf{d} } (\mu)}^p
.\end{equation}

\noindent{\underline{$1.$ The Wasserstein ball DRO setting.}} We proceed in $ 3$ steps. \textbf{Step (i):} we compute the gradient (in the Fréchet sense) of $ \mathcal{L}^{n ,r}_{\W_p}$ whose differentiability is a direct consequence of the differentiability of continuous multilinear mappings and of the differentiability of $T \mapsto \mathbb{E}^{\mu} \big[ \vert T ( X ) \vert^p \big]$ as $ p > 1$. 
Using our notations, we set 
 $
 L^{n, r}_{\W_p} ( T ) 
 :=
 \sum_{i =1}^{n} r^i \phi_i ( T, \cdots, T ) 
 $
which gives the expansion
 $$
 \nabla_T \mathcal{L}^{n , r}_{\W_p} ( T, \lambda ) 
 =
 \sum_{i =1}^{n} \sum_{j =1}^i r^i \nabla_j \phi_i ( T, \cdots, T ) 
 -
 \lambda p \, w_p( T ),
 $$
 where $ \phi_i : \mathbb{L}^p ( \mu )^i \rightarrow \R$ is a continuous symmetric (by Proposition \ref{prop:symm k-th lin deriv}) $i-$linear map and $\nabla_j \phi_i : \mathbb{L}^p( \mu )^{i-1} \rightarrow \mathbb{L}^{p'} ( \mu ) $ is a continuous $(i-1)-$linear map. 

 \noindent{\textbf{Step (ii):}} we now proceed to the construction of $ T^1, \cdots, T^n $ and $ \lambda_1, \cdots, \lambda_n$. We begin by the construction of $T^1$ and $\lambda_1$. Choosing $ T^1 \in \mathbb{L}^p( \mu )$ and $\lambda_1 > 0$ to adjust, we have
$
\nabla_T \mathcal{L}^{n , r} ( T^1, r \lambda_1 ) 
=
r \partial_x \delta_m g ( \mu, X ) 
-
r \lambda_1 p \, w_p ( T^1 )
+
\circ_{\mathbb{L}^{p'} ( \mu ) }  ( r ) 
.$ However, we want them to satisfy $\nabla_T \mathcal{L}^{n , r} ( T^1, r \lambda_1 ) 
=
\circ_{\mathbb{L}^{p'} ( \mu ) }  ( r ) 
$.
Choosing $ T^1$ and $\lambda_1$ such that all terms of order $r$ cancel yields
$$
 \partial_x \delta_m g ( \mu, X ) 
=
\lambda_1 p \, w_p( T^1 ) \, \, \text{and} \,\, 
 \mathbb{E}^{\mu}\left[ \vert T^1 \vert^p \right] = 1 
.$$
This equation is satisfied for $T^1 = T^1_{\W_p}$ defined by equation \eqref{eqdef:first order transport} and $\lambda_1 
= \frac{ \Vert \partial_x \delta_m g \Vert_{\mathbb{L}^{p'} ( \mu )} }{p}.
$

\noindent{\underline{Assume that $ T^1, \cdots, T^k $ and $\lambda_1, \cdots, \lambda_k $ are constructed}}. Let $ T^{k+1} \in \mathbb{L}^p ( \mu )$ and $\lambda_{k+1} \in \R$ to adjust. Define for $ j \leq k +1$,  $ \TT^j := ( T^i )_{1 \leq i \leq j} $, $\Lambda_j := ( \lambda_i )_{1 \leq i \leq j}$, $\bar{\TT}_{j} (r) := S_{j} (r, \TT^{j})$ and $ \bar{\Lambda}_{j}(r) := rS_{j} (r, \Lambda_{j} ) $. We want them to satisfy 
\begin{equation}\label{eq:cond order k+1}
\nabla_T \mathcal{L}^{n, r}_{\W_p} \big( \bar{\TT}_{k+1} (r),  \bar{\Lambda}_{k+1}(r) \big) 
=
\circ_{\mathbb{L}^{p'} ( \mu ) } ( r^{k+1} ) 
.
\end{equation}
By definition of $\mathcal{L}^{n, r}_{\W_p}$ (see equation \eqref{eqdef:lagrangien}), we have 
\begin{align*}
&\nabla_T \mathcal{L}^{n, r}_{\W_p}\big( \bar{\TT}_{k+1} (r),  \bar{\Lambda}_{k+1}(r) \big) =
\nabla_{T}
L^{n, r}_{\W_p} \big(  \bar{\TT}_{k} (r) + r^k T^{k+1} \big) 
-
p
 \bar{\Lambda}_{k+1}(r)  w_p \big( 
 \bar{\TT}_{k+1} (r)
 \big)
.
\end{align*}
Now, with our notations,
$$
\nabla_{T}
L^{n, r}_{\W_p} \big(  \bar{\TT}_{k+1} (r) \big) 
=
\sum_{i =1}^{n} \sum_{j =1}^i r^i \nabla_j \phi_i \big( \bar{\TT}_{k+1} (r), \cdots,   \bar{\TT}_{k+1} (r)\big)
.$$ By $(i-1)-$linearity of $ \nabla_j \phi_i$ and continuity, we get 
\begin{equation}\label{eq:expansion L T_k}
\nabla_{T}
L^{n, r}_{\W_p} \big(  \bar{\TT}_{k+1} (r) \big) 
=
\nabla_{T}
L^{n, r}_{\W_p} \big(\bar{\TT}_{k} (r) \big) 
+
\circ_{\mathbb{L}^{p'} ( \mu ) }  ( r^{k+1} ) 
.\end{equation}
By Assumption, $ p > n $, so the map $ w_p \in C^{n-1} ( \R, \R) $; hence, by Taylor's formula:
$$
w_p ( \bar{\TT}_{k+1} (r) )
=
w_p (  \bar{\TT}_{k} (r) ) + r^k w_p' ( \bar{\TT}_{k} (r)) T^{k +1} + r^k R(r, T)
$$
where $R(r, T) =  T^{k+1} \int_0^1 \big[ w_p' \big(  \bar{\TT}_{k} (r) + t r^k T^{k+1} \big) - w_p' \big( \bar{\TT}_{k} (r) \big)  \big] \mathrm{d}t 
$. By Assumption \hyperref[ass:reformulation 2 gen dist]{${\rm B}_{\W_p}$} \ref{cond:regularity on g dist} and since $ T^{k+1} \in \mathbb{L}^{p'} ( \mu )$, we have $ R(r, T) = \circ_{\mathbb{L}^{p'} ( \mu )} ( 1 ) $. Furthermore, by Assumption, $ p > n $ so $r^k w_p' (\bar{\TT}_{k} (r) ) T^{k +1} = r^{k} (p - 1 ) \vert T^1 \vert^{p-2} T^{k+1} + \circ_{\mathbb{L}^{p'} ( \mu ) }  ( r^{k} ) $. Taking the product we get 
\begin{equation}\label{eq:DL prod w lambda ordre k+1}
\begin{split}
&\bar{\Lambda}_{k+1}(r)
w_p \big(
 \bar{\TT}_{k+1} (r)
 \big)
 \\
 &= \!
\bar{\Lambda}_{k}(r) w_p \big( \bar{\TT}_{k} (r) \big) 
 \!+\! 
 r^{k+1} 
 \big( 
 \lambda_{k+1} w_p \big( \bar{\TT}_{k} (r) \big) 
 +( p-1) \lambda_1 \vert T_1 \vert^{p-2} T^{k+1} 
 \big) 
 \!+\!
 \circ_{\mathbb{L}^{p'} ( \mu ) }  ( r^{k+1} )
 \\
 &=
 \bar{\Lambda}_{k}(r) w_p \big( \bar{\TT}_{k} (r) \big) 
 + 
 r^{k+1} 
 \big( 
 \lambda_{k+1} w_p \big( T^1 \big) 
 +
 ( p-1) \lambda_1 \vert T^1 \vert^{p-2} T^{k+1} 
 \big) 
 +
 \circ_{\mathbb{L}^{p'} ( \mu ) }  ( r^{k+1} )
.
 \end{split}
\end{equation}
Putting equations \eqref{eq:DL prod w lambda ordre k+1} and \eqref{eq:expansion L T_k} together yields, 
\begin{align*}
&\nabla_T \mathcal{L}^{n, r}_{\W_p} \big( \bar{\TT}_{k+1} (r) , \bar{\Lambda}_{k+1}(r) \big) 
\\
&=\!
\nabla_T \mathcal{L}^{n, r}_{\W_p} \big( \bar{\TT}_{k} (r) , \bar{\Lambda}_{k}(r) \big) 
\!-\!
p
 r^{k+1} 
 \big( 
 \lambda_{k+1} w_p( T^1 )
\! + \!
 ( p \!-\! 1 ) \lambda_1 T^{k+1} \vert T^1 \vert^{p-2} 
 \big) 
\!+\!
\circ_{\mathbb{L}^{p'} ( \mu ) }  ( r^{k+1})
\\
&= 
\nabla_{T}
\mathcal{L}^{n, r}_{\W_p} \big(\bar{\TT}_{k} (r) \big) 
-
p
 r^{k+1} 
 \big( 
 \lambda_{k+1} w_p( T^1 )
 +
 ( p-1 ) \lambda_1 T^{k+1} \vert T^1 \vert^{p-2} 
 \big) 
+
\circ_{\mathbb{L}^{p'} ( \mu ) }  ( r^{k+1})
.\end{align*}
We need to identify the term of order $k+1$. Using the multi-linearity of $\nabla_i \phi_j$ and the definition of $L^{n, r}_{\W_p}$ , we easily have the existence of $ V_1, \cdots, V_{k+1} \in \mathbb{L}^{p'} ( \mu ) $ such that for 
$
\nabla_{T}
L^{n, r}_{\W_p} \big( \bar{\TT}_{k} (r) \big) 
=
\sum_{j = 1}^{k+1} r^{j} V_j 
+
\circ_{\mathbb{L}^{p'} ( \mu ) }  ( r^{k+1} )
$.
Also, since $w_p$ is $C^{n-1}$, there exists $ u_1, \cdots, u_{k+1} \in \mathbb{L}^{p'} ( \mu )$, such that for $ 1 \leq i \leq k+1$, $ u_i \in \sigma(T^1, \cdots, T^i )$ and
$$
\bar{\Lambda}_{k}(r) w_p ( \bar{\TT}_{k} (r) ) 
=
\sum_{j = 1}^{k+1} r^{j} u_j 
+
\circ_{\mathbb{L}^{p'} ( \mu ) }  ( r^{k+1} )
.
$$
By construction of $ T^1, \cdots, T^k $ and $ \lambda_1, \cdots, \lambda_k$, we have, 
$$\nabla_T \mathcal{L}^{n, r}_{\W_p} \big( \bar{\TT}_{k} (r) , \bar{\Lambda}_{k}(r) \big) = \circ_{ \mathbb{L}^{p'} ( \mu )} ( r^k ) $$ hence,
$$ \nabla_T
\mathcal{L}^{n, r}_{\W_p} \big(\bar{\TT}_{k} (r) ,  \bar{\Lambda}_{k}(r) \big) 
=
r^{k+1} \big( V_{k+1} - p u_{k+1} \big) 
+
\circ_{\mathbb{L}^{p'} ( \mu ) }  ( r^{k} )
.$$
With the last equation along with equation \eqref{eq:DL prod w lambda ordre k+1}, we see that choosing $\lambda_{k+1} $ and $ T^{k+1}$ such that 
\begin{equation*}
p
 \big( 
 \lambda_{k+1} w_p( T^1 )
 +
 ( p-1 ) \lambda_1 T^{k+1} \vert T^1 \vert^{p-2} 
 \big)
 =
 V_{k+1} - p u_{k+1} 
\end{equation*}
Now, the condition $ \Vert \bar{\TT}_{k+1} (r) \Vert_{\mathbb{L}^p (\mu)}^p = 1 + \circ ( r^{k} ) $ needs to be fulfilled. Following the same computations we did before, we have 
\begin{align*}
\Vert  \bar{\TT}_{k+1} (r) \Vert^p_{\mathbb{L}^p ( \mu )} 
&= 
\Vert \bar{\TT}_{k} (r) \Vert^p_{\mathbb{L}^p ( \mu )} + r^k \E^{\mu} \Big[ w_p \big( \bar{\TT}_{k} (r) \big) T^{k+1} \Big] + \circ ( r^k ) 
\\
&=
\Vert \bar{\TT}_{k} (r) \Vert^p_{\mathbb{L}^p ( \mu )} + r^k \E^{\mu} \Big[ w_p \big( T^1 \big) T^{k+1} \Big] + \circ ( r^k ) 
.\end{align*}
And, by construction, $\Vert \bar{\TT}_{k} (r) \Vert^p_{\mathbb{L}^p ( \mu )} = 1 + \circ ( r^{k-1} )$, so by uniqueness of the development and since $ p > n$, there exists $\alpha \in \R $ such that $\Vert \bar{\TT}_{k} (r) \Vert^p_{\mathbb{L}^p ( \mu )} = 1 + r^{k} \alpha + \circ ( r^{k} )$ hence, 
$$
\Vert  \bar{\TT}_{k+1} (r) \Vert^p_{\mathbb{L}^p ( \mu )} 
= 
 1 + r^{k} \alpha 
+ p r^k \E^{\mu} \Big[ w_p \big( T^1 \big) T^{k+1} \Big] + \circ ( r^k ) 
.$$
Hence, in order to have equation \eqref{eq:cond order k+1} satisfied, it is sufficient for $ \lambda_{k+1} $ and $T^{k+1}$ to satisfy
$$
\left\{
 \begin{array}{ll}
p
 \big( 
 \lambda_{k+1} w_p( T^1 )
 +
 ( p-1 ) \lambda_1 T^{k+1} \vert T^1 \vert^{p-2} 
 \big)
 =
 V_{k+1} - p u_{k+1} 
 \\
 p\E^{\mu} \Big[ w_p \big( T^1 \big) T^{k+1} \Big] = -\alpha
 \end{array}
\right.
.$$
Which is satisfied by 
$$
\left\{
 \begin{array}{ll}
 \lambda_{k+1} 
 =
 \frac{1}{p} \Big( \E^{\mu} \big[ T^1 ( V_{k+1} - p u_{k+1} ) \big] + (p-1) \alpha \lambda_1 \Big)
 \\
T^{k+1} = \frac{1
}
{ p (p-1) \lambda_1 \vert T^1 \vert^{p-2} 
} 
\big( V_{k+1} - p u_{k+1} - p \lambda_{k+1} w_{p}(T^1) \big)
 \end{array}
\right.
.$$
The last recurrent equation is explicit since by construction, $V_{k+1}, u_{k+1}$ and $\alpha$ do not depend on $T_{k+1}$ or $\lambda_{k+1}$. Furthermore, by the Assumption \hyperref[ass:exist approx saddle point wasserstein]{${\rm C}_{\W_p}$} \ref{cond:div by gradient wasserstein}, $T^{k+1} \in \mathbb{L}^{p} ( \mu ) $.

\noindent{ \bf{Step (iii)}:} we now use $ T^1, \cdots, T^n $ to construct a $ \circ ( r^{n} ) $ maximiser of the problem $\sup_{ \substack{T \in \mathbb{L}^p_{\W_p} ( \mu ) \\ \Vert T \Vert_{\mathbb{L}^p ( \mu ) } \leq 1} } L^{n, r}_{\W_p} \big( T \big) $. 
First, we prove that $ T \mapsto L^{n, r}_{\W_p} \big( T \big) $ is a concave function. For $ h \in \mathbb{L}^p ( \mu ) $, since $L^{n, r} $ is a polynomial function it is twice differentiable; for $T \in \mathbb{L}^p ( \mu ) $, denote by $D^{2} L^{n, r}_{\W_p} ( T ) $ its second differential computed at $T$, seen as a bi-linear mapping. We have
\begin{align*}
 D^{2} L^{n, r} (T)  ( h, h ) 
&=
r^2 
\Big( 
\mathbb{E}^{\mu} \big[ \partial^2_{xx} \delta_m g( \mu, X ) h ( X )^2 \big]
\!\! +\!\!
\mathbb{E}^{\mu^{\otimes 2}} \big[ \partial^2_{x^1 x^2} \delta^2_m g ( \mu, X^1, X^2 ) h ( X^1 ) h ( X^2 ) \big]
\Big)
\\
& \,\,\,\,\,\, +
r^3 R_T ( h, h ) 
\end{align*}
where $R_T$ is a polynomial function on $T$ , and $R_T$ is bilinear and continuous with respect to $h$. By Condition \ref{cond:regularity on g dist} of Assumption \hyperref[ass:reformulation 2 gen dist]{$\textbf{\rm B}_{\W_p}$} on $g$, we have 
$$
\sup_{ \Vert T \Vert_{\mathbb{L}^p ( \mu )} \leq 1, \Vert h \Vert_{\mathbb{L}^p ( \mu )} \leq 1} 
 R_T ( h, h ) 
 =
 C < + \infty
.$$
 Now, remark that by strict concavity Assumption \hyperref[ass:exist approx saddle point wasserstein]{$\textbf{\rm C}_{\W_p}$} \ref{cond:convexity wasserstein}, we have
 $$
 \sup_{
 \substack{ \Vert T \Vert_{\mathbb{L}^p ( \mu )} \leq 1 \\ \Vert h \Vert_{\mathbb{L}^p ( \mu )} \leq 1 }} 
\mathbb{E}^{\mu}\big[ \partial^2_{xx} \delta_m g( \mu, X ) h ( X )^2 \big]
+
\mathbb{E}^{\mu^{\otimes 2 }}\big[ \partial^2_{x^1 x^2} \delta^2_m g ( \mu, X^1, X^2 ) h ( X^1 ) h ( X^2 ) \big]
\leq 
-\kappa
.$$
So, for $r$ small enough, we have 
$$
 \sup_{ \substack{ \Vert T \Vert_{\mathbb{L}^p ( \mu )} \leq 1 \\ \Vert h \Vert_{\mathbb{L}^p ( \mu )} \leq 1} }
\mathbb{E}^{\mu}\big[ \partial^2_{xx} \delta_m g( \mu, X ) h ( X )^2 \big]
+
\mathbb{E}^{\mu^{\otimes 2}} \big[ \partial^2_{x^1 x^2} \delta^2_m g ( \mu, X^1, X^2 ) h ( X^1 ) h ( X^2 ) \big]
+
r C 
< 0
.$$
Combining this inequality with the computation of the order two differential, we get that $L^{n, r}_{\W_p}$ is concave for $r$ small enough. Set $
\hat{T}_n (r) := \frac{\bar{\TT}_{n} (r)}{\Vert\bar{\TT}_{n} (r) \Vert_{\mathbb{L}^p ( \mu )}} 
$, by concavity
$$
\sup_{ \Vert T \Vert_{\mathbb{L}^p ( \mu )} \leq 1}
L^{n, r}_{\W_p} ( T )
\leq 
L^{n, r}_{\W_p} (\hat{T}_n (r) )
+
\sup_{ \Vert T \Vert_{\mathbb{L}^p ( \mu )} \leq 1 }
\mathbb{E}^{\mu}\Big[ 
\nabla_T L^{n, r}_{\W_p} (\hat{T}_n (r) )
\big( 
T - \hat{T}_n (r)
\big)
\Big]
.$$ 
We now want to prove that $\sup_{ \Vert T \Vert_{\mathbb{L}^p ( \mu )} \leq 1 }
\mathbb{E}^{\mu}\Big[ 
\nabla_T L^{n, r}_{\W_p} (\hat{T}_n (r) )
\big( 
T - \hat{T}_n (r)
\big)
\Big] = \circ(r^n)$. First, by Hölder's inequality,
$$
\sup_{ \Vert T \Vert_{\mathbb{L}^p ( \mu )} \leq 1 } \!\!\!\!\!
\mathbb{E}^{\mu}\Big[ 
\nabla_T L^{n, r}_{\W_p} (\hat{T}_n (r) )
\big( 
T \!- \hat{T}_n (r)
\big)
\Big]
\!
=
\!
\Vert
\nabla_T L^{n, r}_{\W_p} ( \hat{T}_n (r) )
\Vert_{\mathbb{L}^{p'} ( \mu )}
- 
\mathbb{E}^{\mu}\big[
\nabla_T L^{n, r}_{\W_p} (\hat{T}_n (r))
\hat{T}_n (r)
\big]
.$$ Now, recall that, by construction, $
\Vert \bar{\TT}_{n} (r)  \Vert^p_{\mathbb{L}^p ( \mu )}
=
1 
+
\circ \left( r^{n-1} \right) 
$, and that $\nabla_T L^{n, r}_{\W_p}$ is a polynomial mapping where $k-1$ linear terms are multiplied by $ r^k$. So, a quick computation proves that 
$$
\nabla_T L^{n, r}_{\W_p} ( \hat{T}_n (r) )=
\nabla_T L^{n, r}_{\W_p} \big( \bar{\TT}_{n} (r)  \big) 
+
\circ_{\mathbb{L}^{p'} ( \mu ) }  ( r^n ).
$$
Also, by construction, 
$
\nabla_T
\mathcal{L}^{n, r}_{\W_p} \big(  \bar{\TT}_{n} (r)  ,  \bar{\Lambda}_{n}(r) \big) 
=
\circ_{\mathbb{L}^{p'} ( \mu ) }  ( r^n )
$
which can be written as 
$$
\nabla_T
L^{n, r}_{\W_p} \big(  \bar{\TT}_{n} (r)  \big) 
=
p
 \bar{\Lambda}_{n}(r)
w_p \big( 
 \bar{\TT}_{n} (r) 
\big)
 +
\circ_{\mathbb{L}^{p'} (\mu) }( r^{n} ) 
.
$$
Since $ \lambda_1 > 0$, for $r$ small enough, $\vert 
  \bar{\Lambda}_{n}(r)
\vert =
  \bar{\Lambda}_{n}(r) $, hence, 
  $$\Vert
\nabla_T L^{n, r}_{\W_p} ( \bar{\TT}_{n})
\Vert_{\mathbb{L}^{p'} ( \mu )} = 
p
 \bar{\Lambda}_{n}(r)
\Vert \hat{T}_n (r)^{p-1} \Vert_{\mathbb{L}^{p'} ( \mu )}
$$
and 
$$
\mathbb{E}^{\mu}\Big[
\nabla_T L^{n, r}_{\W_p} (\bar{\TT}_{n} (r) )
\big(
\bar{\TT}_{n} (r) 
\big)
\Big]
=
\bar{\Lambda}_{n}(r)
\Vert \bar{\TT}_{n} (r)  \Vert^p_{\mathbb{L}^{p} ( \mu )}
.$$ Putting all of this together yields 
\begin{align*}
&\Vert
\nabla_T L^{n, r}_{\W_p} ( \hat{T}_n (r) )
\Vert_{\mathbb{L}^{p'} ( \mu )} 
- 
\mathbb{E}^{\mu}\Big[
\nabla_T L^{n, r}_{\W_p} ( \bar{\TT}_{n} (r)  )
\bar{\TT}_{n} (r) 
\Big]
\\
&=
p\bar{\Lambda}_{n}(r)
\big( 
\Vert \bar{\TT}_{n} (r)  \Vert_{\mathbb{L}^{p} ( \mu )}
-
\Vert \bar{\TT}_{n} (r)  \Vert^p_{\mathbb{L}^{p} ( \mu )}
\big)
\\
&=
p \bar{\Lambda}_{n}(r)
\big( 
( 1 + \circ (r^n ) )^{1/p}
-
( 1 + \circ (r^n ) )
\big)
=
\circ( r^n )
.\end{align*} 

\noindent{\underline{$2.$ The adapted Wasserstein ball DRO setting.}} The proof is essentially the same.

\noindent\textbf{Step (i):} we compute the gradient with respect to $ T_1 $ and $T_2$ of $\mathcal{L}^{n, r}_{\W_p^{\rm ad}}$. We can identify $\nabla_{T_1} \mathcal{L}^{n, r}_{\W_p^{\rm ad}}$ as an element of $\sigma(X_1) \cap \mathbb{L}^{p'} ( \mu )$. We have
$$
\nabla_{T_1} \mathcal{L}^{n, r}_{\W_p^{\rm ad}} ( T, \lambda ) 
\! = \! \nabla_{T_1} L^{n, r}_{\W_p^{\rm ad}} ( T, \lambda ) 
-\!
p \lambda w_p ( T_1 ) 
\,\text{and} \,
\nabla_{T_2} \mathcal{L}^{n, r}_{\W_p^{\rm ad}} ( T, \lambda ) 
\!=\!
\nabla_{T_2} L^{n, r}_{\W_p^{\rm ad}} ( T, \lambda ) 
\! -\!
p \lambda w_p ( T_2 )
$$
with 
$
\nabla_{T_1} L^{n, r}_{\W_p^{\rm ad}} ( T, \lambda ) 
 \! = \! 
\sum_{ k = 1}^n 
\! r^k \hat{K}_{k, 1} ( T ) (X_1)
$ and $
\nabla_{T_2} L^{n, r}_{\W_p^{\rm ad}} ( T, \lambda ) 
\!=\!
\sum_{ k = 1}^n 
 r^k \hat{K}_{k, 2} ( T )
$
where 
\begin{equation*}
\begin{split}
 \hat{K}_{k, 1} ( T ) &:=
 \Sum_{
\substack{ 1 \leq \ell \leq k \\
 i \in (\mathbb{N}^*)^\ell , j \in \N^\ell
\\
\vert i \vert_1 = k
, j \leq i
\\
1 \leq a \leq \ell
}}
\frac{ 1 }{\ell! i!}
j_a T_1( X_1 )^{j_a - 1} 
\E_1^{\mu} \Big[ T_2(X)^{i_a - j_a}
\Kc_{\ell, a, i, j} ( T, X)
\Big]
\\
 \hat{K}_{k, 2} ( T ) &:=
\sum_{ \substack{
1 \leq \ell \leq k \\
 i \in (\mathbb{N}^*)^\ell , j \in \N^\ell
\\
\vert i \vert_1 = k
, j \leq i
\\
1 \leq a \leq \ell
}} 
\!\!\!
\frac{ 1 }{\ell! i!}
(i_a - j_a )
T_2(X)^{i_a - j_a - 1}
T_1( X_1 )^{j_a } 
\Kc_{\ell, a, i, j} ( T , X)
\end{split}
\end{equation*}
where 
$$\Kc_{\ell, a, i, j} ( T , x)
:=
\hat{\E}^{\mu^{\otimes \ell - 1}}
\Big[
\partial^{i}_{ \mathbf{x}_1^j \mathbb{x}_2^{i-j} }
 \delta_m^\ell g \big( \mu, ( \hat{X}^{-a}, x) \big)
 \prod_{ \substack{ z = 1 \\ z \neq a }}^{\ell} 
 T_1 ( \hat{X}^z _1 )^{i_z} 
 T_2 ( \hat{X}^z )^{i_z - j_z}
\Big]
$$
with $( \hat{X}^{-a}, x) := ( \hat{X}^1, \cdots, \hat{X}^{a-1}, x ,\hat{X}^{a+1}, \cdots, \hat{X}^{\ell -1} )$.

\noindent{\textbf{Step (ii)}:} we now construct $T^1_1, \cdots, T^n_1, T^1_2, \cdots T^n_2$ and $ \lambda_1 ,\cdots, \lambda_n$. We begin by the construction of $T^1_1$, $T^1_2$ and $\lambda_1$. Following the steps of the previous case yields
\begin{align*}
&\nabla_{T_1} \mathcal{L}^{n, r}_{\W_p^{\rm ad}} \left( T_1, T_2 , \lambda \right) 
=
r \E^\mu \big[ \partial_{x_1} \delta_m g ( \mu, X ) \vert X_1 \big] 
-
p\lambda w_p ( T_1 ) 
+
\circ_{\mathbb{L}^{p'} (\mu_1) } ( r ) 
\end{align*}
and
\begin{align*}
\nabla_{T_2} \mathcal{L}^{n, r}_{\W_p^{\rm ad}} ( T_1, T_2 , \lambda ) 
=
r \partial_{x_2} \delta_m g \left( \mu, X \right) 
-
p \lambda w_p ( T_2)
+
\circ_{\mathbb{L}^{p'} (\mu) } ( r ) 
.
\end{align*}
Hence it is clear that it is sufficient to take $ T^1_{1}, T^1_{2}$ and $\lambda_1$ satisfying
$$
\left\{
 \begin{array}{ll}
\E_1^\mu [ \partial_{x_1} \delta_m g ( \mu, X ) ] 
=
p
\lambda_1 w_p ( T^1_1 ) 
\\
 \partial_{x_2} \delta_m g ( \mu, X ) 
=
p
\lambda_1 w_p ( T^1_2 ) 
\\
\Vert T^1 \Vert_{\mathbb{L}^p_{\rm ad} ( \mu )}
=
1
\end{array}
\right.
.$$
This is satisfied for $ T^1 = T^1_{\W_p^{\rm ad}}$ defined by equation \eqref{eqdef:first order transport} and $
p\lambda_1 
= \Vert \partial_x^{\rm ad} \delta_m g \Vert_{\mathbb{L}^p_{\rm ad} ( \mu ) }
$. Now assume that $T^1_1, \cdots, T^k_1$, $T^1_2, \cdots, T^k_2$ and $ \lambda_1, \cdots, \lambda_k$ are constructed. The construction of the rest is very similar. We want to find $ T^{k+1}_1, T^{k+1}_2$ and $\lambda_{k+1} $ such that 
$$
\left\{ \begin{array}{ll}
&\nabla_{T} \mathcal{L}^{n, r}_{\W_p^{\rm ad}} \big( S_k ( r, \TT^{k+1}) ,  r S_k ( r, \Lambda_{k+1}) \big) 
=
\circ_{\mathbb{L}^{p'} (\mu_1) } ( r^{k+1} )
\\
&\nabla_{T_2} \mathcal{L}^{n, r} \big( S_k ( r, \TT^{k+1}) ,  r S_k ( r, \Lambda_{k+1})  \big) 
=
\circ_{\mathbb{L}^{p'} (\mu) } ( r^{k+1} )
\\
& \Big\Vert S_k ( r, \TT^{k+1}) \Big\Vert_{ \mathbb{L}^p_{\rm ad} (\mu)}^p
=
1
+
\circ ( r^{k} )
\end{array} \right.
$$
where $ \TT^{k+1} := ( T^i )_{1 \leq i \leq k+1} $ and $ \Lambda_{k+1} := ( \lambda_i )_{1 \leq i \leq k +1}$.  By the same computation as in the step $(iii)$, in the context of the Wasserstein ball, as done in the setting $\mathbf{d} = \W_p$, there exists $ U_1, U_2 \in \mathbb{L}^{p'}( \mu ), \alpha \in \R$, depending only on $T^1_1, \cdots, T^k_1$, $T^1_2, \cdots, T^k_2$, $\lambda_1, \cdots, \lambda_k $ such that 
$$
\left\{
 \begin{array}{ll}
p \big( \lambda_{k+1} w_p( T^1_1 ) + \lambda_1 \left( p - 1 \right) \vert T^1_{1} \vert^{p-2} T^{k+1}_1 \big)
 =
 U_1
 \\
p \big( \lambda_{k+1} w_p( T^1_2 ) + \lambda_1 \left( p - 1 \right) \vert T^1_{2} \vert^{p-2} T^{k+1}_2 \big)
 =
 U_2
 \\
p
\mathbb{E}^{\mu}
\big[ w_p( T^1_1) T^{k+1}_1 
+
w_p(T^1_2) T^{k+1}_2 
\big]
=
\alpha
 \end{array}
\right.
$$
for some $U_1, U_2$ and $\alpha$. The rest of the proof is the same.
\ep

\subsection{Proof of Proposition \ref{prop:mart expansion}}
We first prove an upper bound, then we prove the lower bound. 

\noindent {\bf Step (i):} We prove the upper bound: 
\begin{equation}\label{ineq:upper bound mart}
G^{\rm M}_{\rm ad} (r) \leq 
 L^{n,r}_{\W_p^{\rm ad}} \big(  S_{n} ( r, \TT^n ) \big) + \circ(r^n).
\end{equation}
Let $ ( T^i_1 , T_2^i , h_i ,\lambda_i )_{1 \leq i \leq n }$ be defined by Assumption \ref{ass: existence approx sad-point Mart adapted}, first notice that by weak duality 
\begin{align*}
G^{\rm M}_{\rm ad} (r) 
=
\sup_{ \mu' \in B^{ \rm M }_{\W^{\rm ad}_p} ( \mu, r ) } g ( \mu' ) 
 &=
 \sup_{ \mu' \in B_{\W^{\rm ad}_p} ( \mu, r ) } 
 \inf_{ h \in \mathbb{L}^{p'} ( \mu ) } 
 \big( g ( \mu' ) + \E^{\mu'}[ h^{\otimes} ] \big)
 \\
 &\leq 
 \inf_{ h \in \mathbb{L}^{p'} ( \mu ) } 
 \sup_{ \mu' \in B_{\W^{\rm ad}_p} ( \mu, r )} 
 g ( \mu' ) + \E^{\mu'}[ h^{\otimes} ]
 \\
 &\leq 
 \sup_{ \mu' \in B_{\W^{\rm ad}_p} ( \mu, r )}
 g ( \mu' ) + \E^{\mu'} \big[ S_n (r, \mathbf{h}^{\otimes}) \big] \,\, \text{where} \,\, \mathbf{h}^{\otimes}:= ( h_i^{\otimes} )_{1 \leq i \leq n} 
.\end{align*}
However, we see that 
\begin{align*}
 \sum_{i=1}^n r^{i-1} \E^{ \mu^{r, T} } [ h_i^{\otimes} ] 
 &=
 \sum_{i =1}^n 
 r^{i -1}
 \E^\mu \Big[ h_i \big( X_1 + r T_1 ( X_1 ) \big) \big( X_2 - X_1 + r( T_{2} - T_1 ) \big) \Big]
 \\
 &=
 \sum_{i =1}^n 
 r^{i}
 \E^\mu \Big[ h_i \big( X_1 + r T_1 ( X_1 ) \big) ( T_{2} - T_1 ) \Big]
 \,\,\, \text{ since $\mu \in {\rm M} $.}
 \\
 &=
 \sum_{i =1}^n 
 \sum_{k =0}^{n-i}
 r^{i + k }
 \E^\mu \Big[ h_i^{( k)} ( X_1) T_1( X_1)^k \big( T_{2} - T_1(X_1) \big) \Big]
 \\
 &\,\,\,\,\,\,+
 r^n 
 \sum_{i =1}^n 
  \int_0^1
 \E^\mu \Big[ \frac{( 1 - t )^{i-1} }{ (i - 1 )! }
\Delta_{rt T_1 (X_1) } [h_i^{(n-i)}] (X_1)
 ( T_{2} - T_1 )
 \big]
  \mathrm{d}t
.\end{align*}
Now, by Assumptions \hyperref[ass: existence approx sad-point Mart adapted]{${\rm D}_{\rm M}$}, each $h_i \in \Ec_{n-i}$. By the same computations as in the proof of Proposition \ref{prop:approx 1}, we get 
\begin{align*}
&\sup_{ \substack{ T \in \mathbb{L}^p_{\rm ad} (\mu) \\ \Vert T \Vert_{\mathbb{L}^p_{\rm ad} (\mu)} \leq 1} } 
 g ( \mu^{r, T}) + \E^{ \mu^{r, T} } \Big[ \sum_{i=1}^n h_i^{\otimes} r^{i-1} \Big] 
 \\
 &=
 \sup_{ \substack{ T \in \mathbb{L}^p_{\rm ad} (\mu) \\ \Vert T \Vert_{\mathbb{L}^p_{\rm ad} (\mu)} \leq 1} } 
 L^{n, r}_{\W_p^{\rm ad}} (T) + \sum_{i=1}^n \sum_{j =1}^{n-k} \frac{r^{k + j} }{k!} \E^{ \mu} \big[ h^{(k)}_i T_1 (X_1)^k ( T_2 - T_1 (X_1) ) \big] 
 +
 \circ ( r^n ) 
 \\
 &=
 \sup_{ \substack{ T \in \mathbb{L}^p_{\rm ad} (\mu) \\ \Vert T \Vert_{\mathbb{L}^p_{\rm ad} (\mu)} \leq 1} } 
 L^{n, r}_{\W_p^{\rm ad}} (T) + \Gamma^{n, r} ( \mathbf{h}, T )
 +
 \circ ( r^n )
\end{align*}
where $\mathbf{h} = ( h_1, \cdots, h_n ) $ and $\Gamma$ is defined by equation \eqref{eqdef:Gamm and L martingale}. Again, using the concavity Assumption \hyperref[ass: existence approx sad-point Mart adapted]{${\rm D}_{\rm M}$} \ref{cond: concavity martingale} and doing the same computations as in Proposition \ref{prop:expansion order n}, we get 
\begin{equation}\label{eq:upp bound martingale gamma }
 \sup_{\substack{ T \in \mathbb{L}^p_{\rm ad} (\mu) \\ \Vert T \Vert_{\mathbb{L}^p_{\rm ad} (\mu)} \leq 1} } 
  L^{n, r}_{\W_p^{\rm ad}}  ( T ) + \Gamma^{n, r} ( h, T ) 
 \leq 
 L^{n,r}_{\W_p^{\rm ad}} \big( S_{n} ( r, \TT^n ) \big) + \Gamma^{n, r} ( \mathbf{h}, S_{n} ( r, \TT^n ) ) + \circ(r^n)
\end{equation}
where $ \TT^n := ( T^i )_{1 \leq i \leq n} $. Since for all $ 1 \leq i \leq n $, $\E^\mu_1 [ T^i_2 ] = T_1^i $, we have 
$
\Gamma^{n, r} ( h,S_{n} ( r, \TT^n ) )
=
0
$, which, when combined with inequality \eqref{eq:upp bound martingale gamma } gives the estimate \eqref{ineq:upper bound mart}.

\noindent{Step (ii):} we prove the lower bound 
\begin{equation}\label{ineq:lower bound mart}
G^{\rm M}_{\rm ad} (r) \geq 
 L^{n,r}_{\W_p^{\rm ad}} \big(  S_{n} ( r, \TT^n ) \big) + \circ(r^n).
\end{equation}
Let $\big( (T^{k, \varepsilon}_1)_{1 \leq k \leq n } \big)_{\varepsilon > 0}$ be a family of functions, such that for $ 1 \leq k \leq n $, $T_1^{k, \varepsilon}$ is a $C^\infty $ compactly supported function satisfying $T_1^{k, \varepsilon} \xrightarrow[\varepsilon \rightarrow 0]{} T_1^k $. Let $T_2^{k, \varepsilon} = T^k_2 - T^k_1 + T^{k, \varepsilon}_1$. In order to prove the converse inequality, set 
 $$
 \mu_{r, \varepsilon}
 :=
 \mu \circ \big( X + \frac{r}{C_{r, \varepsilon} } S_n (r, T^{k, \varepsilon} ) \big)^{-1}
 \,\,\,\text{and}\,\,\,
 C_{r, \varepsilon} := \Vert S_n (r, T^{k, \varepsilon} ) \Vert _{\mathbb{L}^p ( \mu )}
.$$
Clearly, $\mu_{r, \varepsilon}$ is a martingale measure. Furthermore, since all $(T^k_1)_{1 \leq k \leq n}$ are smooth and compactly supported, the coupling $ \pi_{r, \varepsilon} := \mu \circ \big( X, X +\frac{r}{C_{r, \varepsilon} } S_n (r, T^{k, \varepsilon} ) \big)^{-1} $, is bi-causal; finally, by construction, $
\W^{\rm ad}_p(\mu,\mu_{r, \varepsilon}) \leq r
.$ Hence, we have the inequality 
 $
 G^{\rm M}_{\rm ad} (r ) \geq g ( \mu_{r, \varepsilon} )
.$ Now, since for $ 1 \leq i \leq n $ , $\E^\mu_1 [ T^i_2] =T^i_1 $, we have $g ( \mu_{r, \varepsilon} ) \xrightarrow[\varepsilon \rightarrow 0]{} g(\mu_r)$ where 
 $$
 \mu_{r}
 :=
 \mu \circ \big( {\rm Id} + \frac{r}{C_{r} }S_{n} ( r, \TT^n ) \big)^{-1}
 \,\,\,\text{and}\,\,\,
 C_{r} := \Vert S_{n} ( r, \TT^n ) \Vert _{\mathbb{L}^p ( \mu )}
.$$
Again, using the expansion with linear functional derivative, and the classical expansion, we have 
 $$
 g ( \mu_{r} )
 =
 g ( \mu ) 
 +
  L^{n, r}_{\W_p^{\rm ad}} \Big(\frac{1}{C_r} S_{n} ( r, \TT^n ) \Big) 
 +
 \circ ( r^n ) 
.$$
Since $ C_r = 1 + \circ ( r^{n-1} ) $, we get 
$$
  L^{n, r}_{\W_p^{\rm ad}}  \Big(\frac{ S_{n} ( r, \TT^n ) }{C_r}  \Big) 
 =
  L^{n, r}_{\W_p^{\rm ad}}  \big( S_{n} ( r, \TT^n ) \big) + \circ (r^{n} ) 
.$$
Again, the same computations as in the proof of Proposition \ref{prop:expansion order n} yield
$  L^{n, r}_{\W_p^{\rm ad}}  \Big( \sum_{i=1}^{n} r^{i-1} \frac{ T^{i} }{C_r} \Big) =  L^{n, r}_{\W_p^{\rm ad}}  \big( \sum_{i=1}^{n} r^{i-1} T^i \big) + \circ ( r^{n} ) 
$. Combining the last equality with the inequality \eqref{ineq:lower bound mart}, we get 
$$
  L^{n, r}_{\W_p^{\rm ad}}  \big( S_{n} ( r, \TT^n ) \big) + \circ (r^{n} ) \leq G^{\rm M}_{\rm ad} (r) \leq  L^{n, r}_{\W_p^{\rm ad}}  \big( S_{n} ( r, \TT^n ) \big) + \circ (r^{n} ) 
.$$
\ep

\subsection{Explicit computations of second-order expansion}

In this subsection we will explicitly compute the order $2$ expansion for both functions, $G^{ \mathbf{d} } (r) := \sup_{ \mu' \in B^{\mathbf{d}}( \mu, r) } g( \mu' )$ for $ \mathbf{d} \in \{ \W_p^{\rm ad}, \W_p \}$ and $G^{ \rm M } $. For the sake of clarity, and since we only deal with a maximum of two copies of $X \sim \mu$, we drop the notation $ X^1, \cdots, X^n$ which we replace by $X$ and $\hat{X}$.

\noindent{}\textbf{Order $2$ Expansion of $G^{ \W_p}$ at $0$.} Following the proof of Proposition \ref{prop:expansion order n}, and the symmetry property of Proposition \ref{prop:symm k-th lin deriv}, we need to find $T^1$, $T^2$, $\lambda_1$ and $\lambda_2$ such that 
\begin{align*}
\nabla_T L^{2, r}_{\W_p}\!\!
&=
\! r \partial_x \delta_m g 
\! +\!
r^2 \Big( 
 \big( \partial_{xx} \delta_m g \big) T^1
\! \! + \!
\hat{\E}^{\mu^{\otimes 2}}
\Big[ \big(
\partial_{x \hat{x} } \delta^2_m g ( \mu,\hat{X}, X ) 
 \big)
 T^1 ( \hat{X} ) \Big]
 \Big)
 \\
 &
 \hspace{0.3 cm 
 }-
p ( \lambda_1 + r \lambda_2 ) w_p( T^1 + r T^2 )
 \\
 &=
 \circ ( r^2 ) 
\end{align*}
and 
$$
\E^\mu \big[ \vert T^1 + r T^2 \vert^p \big] 
=
1 + \circ ( r ) 
.$$
We already showed that 
$$
T^1 = \frac{ w_{p'} \big( \partial_x \delta_m g ( \mu, X ) \big) }{ \Vert \partial_x \delta_m g \Vert_{\mathbb{L}^{p'} ( \mu ) }^{p'/p } }
\,\, \text{and} \,\, 
\lambda_1 = \frac{ \Vert \partial_x \delta_m g \Vert_{\mathbb{L}^{p'} ( \mu )} }{p}.
$$
The equations satisfied by $T_2$ and $\lambda_2$ are 
$$
\left\{
 \begin{array}{ll}
p\lambda_{2} w_p ( T^1 ) + \lambda_1 p ( p - 1 ) \vert T^1 \vert^{p-2} T^{2} 
 =
 \partial_{xx} \delta_m g T_1
 +
 \hat{\E}^{\mu} 
\Big[
 \big( \partial_{x \hat{x}} \delta^2_m g ( \mu,\hat{X}, X )
T^1( \hat{X}) \Big]
 \\
p
\mathbb{E}^{\mu}
\big[ w_p (T^1 ) T^{2}
\big]
=
0
 \end{array}
\right.
.$$
The solution of this system is 
$$
\left\{
 \begin{array}{ll}
T_2
 =
 \frac{1}{ p( p - 1 )\lambda_1 \vert T^1 \vert^{p-2} }
 \Big( 
 \partial_{xx} \delta_m g T^1
+
 \hat{\E}^{\mu^{\otimes 2}} 
\Big[
 \big( \partial_{x \hat{x}} \delta^2_m g ( \mu,\hat{X}, X )
T^1 ( \hat{X} ) \Big]
-
p\lambda_2 w_p( T^1 ) 
\Big) 
 \\
\lambda_2
=
\E^{\mu}
\Big[
T_1 
\Big( 
 \partial_{xx} \delta_m g T_1
 +
 \hat{\E}^{\mu^{\otimes 2}} 
\big[
 ( \partial_{x \hat{x}} \delta^2_m g ( \mu,\hat{X}, X )
T^1 ( \hat{X}) \big]
\Big)
\Big]
 \end{array}
\right.
.$$
Now, using the equation $\mathbb{E}^{\mu}
[ w_p( T^1) T^{2} ] = 0 $, we see that 
$
\E^{\mu} \big[ 
\partial_x \delta_m g ( \mu, X ) T^2 \big] 
=
0
$
which gives 
\begin{align*}
G^{\W_p} \! ( r ) \!
&=\!
g ( \mu ) 
\!+\!
r \E^{\mu} \big[ \partial_x \delta_m g T^1 \big]
\!\!+\!
\frac{r^2}{2}
\Big( \!
\E^{\mu} \!\big[ 
(\partial_{xx} \delta_m g ) T^1( \!X\! )^2 \big]
\!\!+\!
\E^{\mu^{\otimes 2}} \!\big[ (\partial_{x \hat{x}} \delta^2_m g ) T^1 \otimes T^1 \big]\!
\Big)
\! +\!
\circ ( r^2 ) 
.\end{align*}
\ep

\noindent{}\textbf{Order $2$ Expansion of $G^{ \W^{\rm ad}_p}$ at $0$.} Following the proof of Proposition \ref{prop:expansion order n}, we need to find $T^1 := ( T^1_1, T^1_2)$, $T^2 := ( T^2_1, T^2_2) $, $\lambda_1$ and $\lambda_2$ such that 
\begin{equation*}
\begin{split}
&p
\lambda_1 w_p ( T^1_1 )
 =
 \E^{\mu}_1 \big[ \partial_{x_1} \delta_m g \big] 
 \,\,\, , \,\,\,
 p
 \lambda_1 w_p ( T^1_2 )
 =
 \partial_{x_2} \delta_m g
 \,\,\, , \,\,\,
 \E^{\mu} \big[ \vert T^1_{1} \vert^p +\vert T^1_2 \vert^p \big] 
 =
 1
 \\
 &p \lambda_2 w_p ( T^1_1 ) \!+\!\! \lambda_1 p ( p \!-\! 1 ) \vert T^1_1 \vert^{p-2} T^2_1
 \!\!=\! 
 \E^{\mu}_1 [ (\partial_{x_1} \partial_x \delta_m g ) \!\cdot \!T^1 ] 
 \!+\!\!
 \sum_{i = 1}^2
 \E_1^{\mu} \Big[ 
 \hat{\E}^{\mu} 
  \big[ \partial_{x_1 \hat{x}_i} \delta^2_m g ( \mu, X, \hat{X} ) T^1_i \big] 
 \Big]
 \\
 & p \lambda_2 w_p( T^1_2)
 + 
 \lambda_1 p ( p - 1) \vert T^1_2 \vert^{p-2} T_2^2
 =
 (\partial_{x_2} \partial_x \delta_m g )\cdot T^1
 +
 \sum_{i = 1}^2 \hat{\E}^{\mu} \big[ \partial_{x_2 \hat{x}_i} \delta^2_m g ( \mu, X, \hat{X} ) T^1_i \big] 
 \\
 &\E^{\mu} [ w_p ( T^1_1 ) T^2_1
 +
 w_p ( T^1_2 ) T^2_2 ] 
 =
 0
\end{split}
\end{equation*}
From previous computations, we already know $ T^1 $ and $\lambda_1$, which are given by \eqref{eqdef:first order transport} and $\lambda_1 = \frac{\Vert \partial_x^{\rm ad} \delta_m g \Vert_{\mathbb{L}^p_{\rm ad} (\mu )}}{p}$. Furthermore, combining the last equation with the expressions of $T^1_1$ and $T^1_2$, we see that 
$
\E^{\mu} [ (\partial_{x_1} \delta_m g) T_1^2] 
+
\E^{\mu}[ (\partial_{x_2} \delta_m g ) T_{2, 2 } ] 
=
0
$, hence, by a quick computation, 
\begin{align*}
G_{\rm ad} ( r) 
&=
g( \mu ) 
+
L^{2,r} ( T^1_{1} + rT^2_{1} , T^1_{2} + r T^2_{2} ) 
+
\circ ( r^2 ) 
\\
&=
g( \mu ) 
+
L^{2,r} ( T^1_{1}, T^1_{2} ) 
+
\circ ( r^2 )
\end{align*}

\ep 

\noindent{\textbf{Order $2$ expansion of $G^{\rm M}_{\rm ad}$ at $0$.}} We write the first order condition for $n = 2$, 
\begin{equation}
\begin{split}
&\E^{\mu}_1[ \partial_{x_1} \delta_m g ]
-
h_1( X_1 ) 
-
\lambda_1 p w_p (T^1_1)
=
0
\\
& \partial_{x_2} \delta_m g 
+
h_1 (X_1)
-
\lambda_1 p w_p ( T^1_2 ) 
=
0
\\
&\E^{\mu} [ \vert T^1_1 \vert^p ]
 +
\E^{\mu} [ \vert T^1_2 \vert^p ]
=
1
\\
&p \lambda_{2} w_p (T_1^1) + p \lambda_1 ( p - 1 ) \vert T^1_1 \vert^{p-2} T^1_2 
 -
2 T^1_1 h'_1 - h_2
=
\sum_{i = 1}^{2} \E^{\mu}_1 [ (\partial_{x_i x_1} g ) T^1_i ]
\\
&p\lambda_{2} w_p( T^1_2) + p\lambda_1 ( p - 1 ) \vert T^1_2 \vert^{p-2} T^2_2 
 =
\sum_{i = 1}^{2} (\partial_{ x_i, x_2} g ) T^1_i 
+ h'_1 T^1_1 + h_2
\\
&\mathbb{E}^{\mu}
[ w_p ( T^1_1 ) T_1^2 ]
+
\mathbb{E}^{\mu}[ w_p (T^1_2) T^2_{2}
]
=
0
\\
& \E^{\mu}_1[ T^1_{2} ] = T^1_{1} \,\,\,\, \text{and} \,\,\,\,
\E^{\mu}_1 [ T^2_{2} ] = T^2_1 
\end{split}
\end{equation}
A quick computation yields 
\begin{align*}
&p\lambda_1 
=
\Vert \partial_x^{\rm ad } \delta_m g \Vert_{\mathbb{L}^{p'} (\mu) } 
\\
& w_{p'} \big( \E^{\mu}_1 [ \partial_{x_1} \delta_m g ] - h_1 \big)
=
\E^{\mu}_1 [ w_{p'} \big( \partial_{x_2} \delta_m g + h_1 \big) ]
\\
& T^1 = \frac{ w_p' (\partial_{x}^{\rm ad} \delta_m g )} {p\lambda_1} 
\\
&h_2 
=
\frac{
 \sum_{i = 1}^{2} 
 \E^{\mu}_1 [ 
 \frac{(\partial_{x_i x_1} g) T^1_i }{ \vert T^1_1 \vert^{p-2} }
- 
 \frac{ (\partial_{ x_i x_2}g) T^1_i}{ \vert T^1_2 \vert^{p-2} }
 ]
-h'_1 \E_1^{\mu} [ \frac{ 1 }{ \vert T^1_{2} \vert^{p-2} } - \frac{ 1 }{ \vert T^1_1 \vert^{p-2} } ]
}
{
1 + \vert T^1_1 \vert^{p-2} \E_1^{\mu}[ \frac{1}{ \vert T^1_{2} \vert^{p-2} } ]
 }
 \\
&p\lambda_2 
=
\E^{\mu} \Big[ T^1 \cdot \big( D^{2} g T^1 \big) \Big]
\\
& T^2_1 
=
\frac{
\sum_{i = 1}^{2} \E^{\mu}_1 [ (\partial_{x_i x_1} g) T^1_i ]
-
p\lambda_{2} w_p(T^1_{1}) 
-
h'_1 T^1_1
-
h_2}
{p\lambda_1 ( p - 1 ) \vert T^1_{1} \vert^{p-2}}
\\
& T^2_2 
 =
 \frac{
\sum_{i = 1}^{2} (\partial_{x_i x_2} g )T^1_{i} 
+h'_1 T^1_{1} + h_2
-
p\lambda_{2} w_p( T^1_{2} )
}{p\lambda_1 ( p - 1 ) \vert T^1_2\vert^{p-2} }
.\end{align*}
And similar considerations as the ones done for the computation of order two expansions of $G^{\rm ad}$ and $G$ yield
$
G^{\rm M}_{\rm ad} (r) = L^{2, r} ( T_1^1, T_2^1 ) + \circ(r^2)
.$

\normalem

\begin{sloppypar}
\printbibliography
\end{sloppypar}

 \end{document}